\documentclass[a4paper,11pt]{amsart}
\usepackage{amscd}
\usepackage{graphics}
\usepackage{url}
\usepackage{epsfig}
\usepackage{parskip}
\usepackage{enumerate}
\usepackage{lscape}

\usepackage[margin=1.2in]{geometry}

\usepackage{tikz}
\usetikzlibrary{arrows}

\newcommand{\N}{\mathbb{N}}
\newcommand{\Z}{\mathbb{Z}}

\newcommand{\R}{\mathbb{R}}
\newcommand{\C}{\mathbb{C}}

\newcommand{\Hy}{\mathbb{H}}

\newcommand{\cO}{\mathcal{O}}

\newcommand{\la}{\left\langle}
\newcommand{\ra}{\right\rangle}
\newcommand{\tg}{\widetilde{\gamma}}

\theoremstyle{definition}
\newtheorem{thm}{Theorem}[subsection]
\newtheorem{corollary}[thm]{Corollary}
\newtheorem{lemma}[thm]{Lemma}
\newtheorem{prop}[thm]{Proposition}
\newtheorem{defn}[thm]{Definition}

\newtheorem{example}[thm]{Example}

\title[Hyperbolic surface group amalgams]{Abstract commensurability and quasi-isometry classification of hyperbolic surface group amalgams }
\author{Emily Stark}
\date{November 1, 2015}

\begin{document}
 
\begin{abstract}
 Let $\mathcal{X}_S$ denote the class of spaces homeomorphic to two closed orientable surfaces of genus greater than one identified to each other along an essential simple closed curve in each surface. Let $\mathcal{C}_S$ denote the set of fundamental groups of spaces in $\mathcal{X}_S$. In this paper, we characterize the abstract commensurability classes within $\mathcal{C}_S$ in terms of the ratio of the Euler characteristic of the surfaces identified and the topological type of the curves identified. We prove that all groups in $\mathcal{C}_S$ are quasi-isometric by exhibiting a bilipschitz map between the universal covers of two spaces in $\mathcal{X}_S$. In particular, we prove that the universal covers of any two such spaces may be realized as isomorphic cell complexes with finitely many isometry types of hyperbolic polygons as cells. We analyze the abstract commensurability classes within $\mathcal{C}_S$: we characterize which classes contain a maximal element within $\mathcal{C}_S$; we prove each abstract commensurability class contains a right-angled Coxeter group; and, we construct a common CAT$(0)$ cubical model geometry for each abstract commensurability class. 
\end{abstract}

 \maketitle
 
 \section{Introduction}
 
 Finitely generated infinite groups carry both an algebraic and a geometric structure, and to study such groups, one may study both algebraic and geometric classifications. Abstract commensurability defines an algebraic equivalence relation on the class of groups, where two  groups are said to be {\it abstractly commensurable} if they contain isomorphic subgroups of finite-index. Finitely generated groups may also be viewed as geometric objects, since a finitely generated group has a natural word metric which is well-defined up to quasi-isometric equivalence. Gromov posed the program of classifying finitely generated groups up to quasi-isometry.

A finitely generated group is quasi-isometric to any finite-index subgroup, so, if two finitely generated groups are abstractly commensurable, then they are quasi-isometric. Two fundamental questions in geometric group theory are to classify the abstract commensurability and quasi-isometry classes within a class of finitely generated groups and to understand for which classes of groups the characterizations coincide.  
 
A basic and motivating example is the class of groups isomorphic to the fundamental group of a closed orientable surface of genus greater than one. These groups act properly discontinuously and cocompactly by isometries on the hyperbolic plane, hence all such groups are quasi-isometric. In addition, every surface of genus greater than one finitely covers the genus two surface, so all groups in this class are abstractly commensurable. In particular, the quasi-isometry and abstract commensurability classifications coincide in this setting. Free groups, which may be realized as the fundamental group of surfaces with non-empty boundary, exhibit the same behavior; there is a unique quasi-isometry and abstract commensurability class among non-abelian free groups.  
 
In this paper, we present a complete solution to the quasi-isometry and abstract commensurability classification questions within the class $\mathcal{C}_S$ of groups isomorphic to the fundamental group of two closed orientable surfaces of genus greater than one identified along an essential simple closed curve in each. We prove that there is a single quasi-isometry class within $\mathcal{C}_S$ and infinitely many abstract commensurability classes.

\subsection{Abstract commensurability and quasi-isometry classification}

In Section $3$, we characterize the abstract commensurability classes within $\mathcal{C}_S$. Our classification uses work of Lafont, who proved that spaces obtained by identifying hyperbolic surfaces with non-empty boundary along their boundary components are {\it topologically rigid}: any isomorphism between fundamental groups of these spaces is induced by a homeomorphism between the spaces \cite{lafont} (see also \cite{crisppaoluzzi}). As a consequence, groups in the class $\mathcal{C}_S$ are abstractly commensurable if and only if the corresponding spaces built by identifying two surfaces along an essential closed curve in each have homeomorphic finite-sheeted covering spaces. We use this fact to obtain topological obstructions to commensurability. 

Before stating the full classification theorem, we present two corollaries: the abstract commensurability classification in the case that groups $G_1$ and $G_2$ are the fundamental groups of surfaces identified along separating curves, and the abstract commensurability classification in the case that groups $G_1$ and $G_2$ are the fundamental groups of surfaces identified along non-separating curves. 

 {\bf Corollary \ref{justsep} } 
  {\it If $S_1, S_2, S_3, S_4$ and $T_1, T_2, T_3, T_4$ are orientable surfaces of genus greater than or equal to one and with one boundary component, the $S_i$ are glued along their boundary to form $X_1$, and the $T_i$ are glued along their boundary to form $X_2$, then $\pi_1(X_1)$ and $\pi_1(X_2)$ are abstractly commensurable if and only if, up to reindexing, the quadruples $(\chi(S_1), \ldots, \chi(S_4))$ and $(\chi(T_1), \ldots, \chi(T_4))$ are equal up to integer scale. 
  }

 {\bf Corollary \ref{justnonsep}}
  {\it  If $S_{g_i}$ and $S_{g_i'}$ are orientable surfaces of genus greater than one identified to each other along a non-separating curve in each to form the space $X_i$ for $i=1,2$, then $\pi_1(X_1)$ and $\pi_1(X_2)$ are abstractly commensurable if and only if, up to reindexing, $\displaystyle \frac{\chi(S_{g_1})}{\chi(S_{g_1'})} = \frac{\chi(S_{g_2})}{\chi(S_{g_2'})}$.
  }
  
   The additional condition in the full classification within $\mathcal{C}_S$ given in Theorem \ref{classification} is that a separating curve that divides the surface exactly in half may be replaced by a non-separating curve on the same surface without changing the abstract commensurability class. We use the following notation. If $\gamma$ is an essential simple closed curve on a surface, the number $t(\gamma)$ is equal to one if $\gamma$ is non-separating, and is equal to $\frac{\chi(S_{r,1})}{\chi(S_{s,1})}$ if $\gamma$ separates the surface into two subsurfaces $S_{r,1}$ and $S_{s,1}$ and $\chi(S_{r,1}) \leq \chi(S_{s,1})$. Our full classification theorem is given as follows. 

{\bf Theorem \ref{classification}.} {\it If $G_1, G_2 \in   \mathcal{C}_S$, then $G_1$ and $G_2$ are abstractly commensurable if and only if, up to relabeling, $G_1 \cong \pi_1(S_{g_1})*_{\left\langle a_1 \right\rangle} \pi_1(S_{g_1'})$ and $G_2 \cong \pi_1(S_{g_2})*_{\left\langle a_2 \right\rangle} \pi_1(S_{g_2'})$, the amalgams are given by the monomorphisms $a_i \mapsto [\gamma_i] \in \pi_1(S_{g_i})$ and  $a_i \mapsto [\gamma_i'] \in \pi_1(S_{g_i'})$, and the following conditions hold: 

  \quad \quad (a) \, $\displaystyle \frac{\chi(S_{g_1})}{\chi(S_{g_1'})} = \frac{\chi(S_{g_2})}{\chi(S_{g_2'})}$, \quad \quad  \quad (b) \, $t(\gamma_1) = t(\gamma_2)$, \quad \quad \quad (c) \, $t(\gamma_1') = t(\gamma_2')$.
 }

 \vskip.2in
 
 The quasi-isometry classification within $\mathcal{C}_S$ stands in contrast to the abstract commensurability classification. Groups in the class $\mathcal{C}_S$ act geometrically on a piecewise hyperbolic CAT$(-1)$ space built by identifying infinitely many copies of the hyperbolic plane along geodesic lines in a `tree-like' fashion. The following theorem, proven in Section 4.3, states that all such spaces have the same large-scale geometry; the quasi-isometry classification follows as a consequence. 

{\bf Theorem \ref{bclassification}.}  {\it Let $\mathcal{X}_S$ denote the class of spaces homeomorphic to two closed orientable surfaces of genus greater than one identified along an essential simple closed curve in each. If $X_1, X_2 \in \mathcal{X}_S$ and $\widetilde{X}_1$ and $\widetilde{X}_2$ are their universal covers equipped with a CAT$(-1)$ metric that is hyperbolic on each surface, then there exists a bilipschitz equivalence $\phi:\widetilde{X}_1 \rightarrow \widetilde{X}_2$.}

{\bf Corollary \ref{qiclassification}.}
 {\it If $G_1, G_2 \in \mathcal{C}_S$, then $G_1$ and $G_2$ are quasi-isometric. }

Our approach in the proof of Theorem \ref{bclassification} is to realize $\widetilde{X}_1$ and $\widetilde{X}_2$ as isomorphic cell complexes with finitely many isometry types of convex hyperbolic polygons as cells. We show there is a bilipschitz equivalence between hyperbolic $n$-gons that restricts to dilation on each edge. Thus, there is a well-defined cellular homeomorphism $\widetilde{X}_1 \rightarrow \widetilde{X}_2$ that restricts to a bilipschitz map on each tile, and we prove this extends to a bilipschitz map $\widetilde{X}_1 \rightarrow \widetilde{X}_2$. 

Groups in the class $\mathcal{C}_S$ also admit a CAT$(0)$ geometry, and an alternative approach to the quasi-isometry classification was given by Malone \cite{malone}, who applied the work of Behrstock--Neumann on the bilipschitz equivalence of fattened trees used in the quasi-isometric classification of graph manifold groups \cite{behrstockneumann}.

 The abstract commensurability classes within $\mathcal{C}_S$ are finer than the quasi-isometry classes; there is a unique quasi-isometry class in $\mathcal{C}_S$, and there are infinitely many abstract commensurability classes. Whyte, in \cite{whyte}, proves a similar result for free products of hyperbolic surface groups.
 
 \begin{thm} (\cite{whyte}, Theorem 1.6, 1.7)
 {\it Let $\Sigma_g$ be the fundamental group of a surface of genus $g\geq 2$ and let $m,n \geq 2$. Let $\Gamma_1 \cong  \Sigma_{a_1}*\Sigma_{a_2}* \ldots * \Sigma_{a_n}$ and $\Gamma_2 \cong \Sigma_{b_1}*\Sigma_{b_2}* \ldots * \Sigma_{b_m}$. Then $\Gamma_1$ and $\Gamma_2$ are quasi-isometric, and $\Gamma_1$ and $\Gamma_2$ are abstractly commensurable if and only if \[\frac{\chi(\Gamma_1)}{n-1} = \frac{\chi(\Gamma_2)}{m-1}.\]   }
\end{thm}

 Similarly, there is a unique quasi-isometry class and infinitely many abstract commensurability classes among the set of fundamental groups of closed graph manifolds, which exhibit a related geometry to groups in $\mathcal{C}_S$ \cite{behrstockneumann}, \cite{neumann}. 
 
 On the other hand, there are many classes of groups for which the quasi-isometry and abstract commensurability classifications coincide. Such classes include non-trivial free products of finitely many finitely generated abelian groups excluding $\Z/2\Z*\Z/2\Z$ \cite{behrstockjanuszkiewiczneumann}, non-uniform lattices in the isometry group of a symmetric space of strictly negative sectional curvature other than the hyperbolic plane \cite{schwartz}, and fundamental groups of $n$-dimensional ($n \geq 3$) connected complete finite-volume hyperbolic manifolds with non-empty geodesic boundary (which must be compact in dimension three) \cite{frigerio}.

 This paper concerns surfaces of negative Euler characteristic. Cashen provides a quasi-isometry classification of the fundamental groups of a disjoint union of (Euclidean) tori glued together along annuli \cite{cashen}.

 \subsection{Analysis of the abstract commensurability classes}
 
  Recent surveys on notions of commensurability are given by Paoluzzi \cite{paoluzzi} and Walsh \cite{walsh}.

 Let $\mathcal{G} \subset \mathcal{C}_S$ be an abstract commensurability class within $\mathcal{C}_S$. A {\it maximal element} for $\mathcal{G}$ is a group $G_0$ that contains every group in $\mathcal{G}$ as a finite-index subgroup. A classic result in the setting of hyperbolic $3$-manifolds is that of Margulis \cite{margulis}, who proved that if $H \leq PSL(2,\C)$ is a discrete subgroup of finite covolume, then there exists a maximal element in the abstract commensurability class of $H$ within $PSL(2, \C)$ if and only if $H$ is non-arithmetic. It follows that the commensurability class of a non-arithmetic finite-volume hyperbolic 3-manifold contains a {\it minimal element}: there exists an orbifold finitely covered by every other manifold in the commensurability class.

 In Section 5.1, we state an alternative formulation of the abstract commensurability classification within $\mathcal{C}_S$, and we show that for abstract commensurability classes $\mathcal{G} \subset \mathcal{C}_S$, the existence of a maximal element $G_0 \in \mathcal{C}_S$ depends on whether the class contains the fundamental group of a surface identified along non-separating curve. 
 
 {\bf Proposition \ref{maximalelement}. } {\it Let $\mathcal{G} \subset \mathcal{C}_S$ be an abstract commensurability class within $\mathcal{C}_S$. There is a maximal element for $\mathcal{G}$ in $\mathcal{C}_S$ if and only if $\mathcal{G}$ does not contain the fundamental group of a surface identified along a non-separating curve to another surface. 
 }

 In Section 5.2, we show that if the abstract commensurability class $\mathcal{G} \subset \mathcal{C}_S$ contains the fundamental group of two surfaces identified along non-separating curves in both surfaces, then there exists a right-angled Coxeter group that is a maximal element for the class. In the remaining case, that the class contains the fundamental group of two surfaces identified along a non-separating curve in exactly one of the surfaces and does not contain the fundamental group of two surfaces identified along a non-separating curve in both, Proposition \ref{maximalelement} shows there is no maximal element in $\mathcal{C}_S$, and the existence of a maximal element outside of $\mathcal{C}_S$ remains open.

 Hyperbolic surface groups are finite-index subgroups of right-angled Coxeter groups. We apply our abstract commensurability classification within $\mathcal{C}_S$ (Theorem \ref{classification}) to prove the following.
 
 {\bf Proposition \ref{ACracg}. } {\it Each group in $\mathcal{C}_S$ is abstractly commensurable to a right-angled Coxeter group.  }
 
 In other words, each abstract commensurability class of a group in $\mathcal{C}_S$ contains a right-angled Coxeter group. In particular, in Section 5.2, we show the fundamental group of two surfaces identified along a separating curve in each and the fundamental group of two surfaces identified along curves of {\it topological type one} (See definition \ref{toptype}) are finite-index subgroups of a right-angled Coxeter group. It is an open question whether each group in $\mathcal{C}_S$ is a finite-index subgroup of a right-angled Coxeter group in the remaining case.

The result in Theorem \ref{classification} is related to the abstract commensurability classification of the following right-angled Coxeter groups introduced by Crisp--Paoluzzi in \cite{crisppaoluzzi} and further studied by Dani--Thomas in \cite{danithomas}. Let \[W_{m,n} = W(\Gamma_{m,n}),\] be the right-angled Coxeter group associated to the graph $\Gamma_{m,n}$, which consists of a circuit of length $m+4$ and a circuit of length $n+4$ which are identified along a common subpath of edge-length $2$. For all $m$ and $n$, the group $W_{m,n}$ is the orbifold fundamental group of a $2$-dimensional reflection orbi-complex $\mathcal{O}_{m,n}$. We show in Lemma \ref{orbicovers} that for all $m$ and $n$, $\mathcal{O}_{m,n}$ is finitely covered by a space consisting of two hyperbolic surfaces identified along non-separating essential simple closed curves. Conversely, we prove all amalgams of surface groups over homotopy classes of  non-separating essential simple closed curves are finite index subgroups of $W_{m,n}$ for some $m$ and $n$, dependent on the Euler characteristic of the two surfaces. Thus, our theorem extends their result. 
  
\begin{corollary} \label{CP} (\cite{crisppaoluzzi} Theorem 1.1) {\it Let $1\leq m \leq n$ and $1 \leq k \leq \ell$. Then $W_{m,n}$ and $W_{k,\ell}$ are abstractly commensurable if and only if $\frac{m}{n} = \frac{k}{\ell}$.  }
\end{corollary}
   
Moreover, in Proposition \ref{CPAC}, we apply our abstract commensurability classification to prove that if $G \in \mathcal{C}_S$, then $G$ is abstractly commensurable to $W_{m,n}$ for some $m$ and $n$ if and only if $G$ is the fundamental group of two surfaces identified to each other along curves of {\it topological type one} (see Definition \ref{toptype}).
    
A {\it model geometry} for a finitely generated group $G$ is a proper metric space $X$ on which $G$ acts properly discontinuously and cocompactly by isometries. Given an abstract commensurability class $\mathcal{G} \subset \mathcal{C}_S$, one can ask whether there is a common model geometry for every group in $\mathcal{G}$. If $\mathcal{G}$ has a maximal element $G_0$, any model geometry for $G_0$ provides a common model geometry for every group in $\mathcal{G}$. For $\mathcal{G} \subset \mathcal{C}_S$ it is not known, in general, if there is a maximal element for $\mathcal{G}$. Nonetheless, we prove there is a common CAT$(0)$ cubical model geometry for every group in $\mathcal{G}$. 

{\bf Proposition \ref{commoncubing}.} {\it Let $\mathcal{G} \subset \mathcal{C}_S$ be an abstract commensurability class within $\mathcal{C}_S$. There exists a $2$-dimensional CAT$(0)$ cube complex $X$ so that if $G \in \mathcal{G}$, $G$ acts properly discontinuously and cocompactly by isometries on $X$. Moreover, the quotient $X/G$ is a non-positively curved special cube complex. 
   }
   
 Similarly, as described in \cite{moshersageevwhyteI}, one can ask if there is a common model geometry for every group in a quasi-isometry class. It is not known whether all groups in $\mathcal{C}_S$ act properly discontinuously and cocompactly by isometries on the same proper metric space.

\subsection{Outline} In Section 2, we define the spaces $\mathcal{X}_S$ and the class of groups $\mathcal{C}_S$ examined in this paper. Section 3 contains the abstract commensurability classification within $\mathcal{C}_S$. In Section 4, we define a piecewise hyperbolic metric on spaces in $\mathcal{X}_S$, construct a bilipschitz equivalence between the universal covers of any such spaces, and conclude all groups in $\mathcal{C}_S$ are quasi-isometric. Section 5 contains an analysis of the abstract commensurability classes, which includes a description of maximal elements for an abstract commensurability class,  a description of the relation of groups in $\mathcal{C}_S$ to the class of right-angled Coxeter groups, and the construction of the common cubical geometry for all groups in an abstract commensurability class within $\mathcal{C}_S$. 
 
\subsection{Acknowledgments} The author is deeply grateful for many discussions with her Ph.D. advisor Genevieve Walsh. The author wishes to thank Pallavi Dani for pointing out a gap in an earlier version of this paper, and her peers at Tufts University for helpful conversations throughout this work. The author is thankful for very useful comments and corrections from an anonymous referee. This material is partially based upon work supported by the National Science Foundation Graduate Research Fellowship Program under Grant No. DGE-0806676. 
 
 
 \section{Surfaces and the class of groups $\mathcal{C}_S$}



We use $S_{g,b}$ to denote the orientable surface of genus $g$ and $b$ boundary components. The {\it Euler characteristic} of a surface $S_{g,b}$ is $\chi(S_{g,b}) = 2-2g-b$. Unless stated otherwise, we will say ``surface'' to mean a compact, connected, oriented surface. We will typically be interested in surfaces of negative Euler characteristic. 

We say a surface $S$ {\it admits a hyperbolic metric} if there exists a complete, finite-area Riemannian metric on $S$ of constant curvature $-1$ and the boundary of $S$ is totally geodesic: the geodesics in $\partial S$ are geodesics in $S$. A surface $S$ may be endowed with a hyperbolic metric via a free and properly discontinuous action by isometries of $\pi_1(S)$ on the hyperbolic plane $\Hy^2$. 

\begin{thm}
 {\it If $S$ is a surface with $\chi(S)<0$, then $S$ admits a hyperbolic metric. }
\end{thm}

A {\it closed curve} in a surface $S$ is a continuous map $S^1 \rightarrow S$, and we often identify a closed curve with its image in $S$. We use $[\gamma]$ to denote the homotopy class of a curve $\gamma$. A closed curve is {\it essential} if it is not homotopic to a point or boundary component.  An essential closed curve $\gamma$ is {\it primitive} if is not the case that $[\gamma] =[\rho^n]$ for some closed curve $\rho$. A closed curve is {\it simple} if it is embedded. A {\it homotopy class of simple closed curves} is a homotopy class in which there exists a simple closed curve representative. A {\it multicurve} in $S$ is the union of a finite collection of disjoint simple closed curves in $S$. 

If $\gamma$ is a simple closed curve on a surface $S$, the surface obtained by {\it cutting} $S$ along $\gamma$ is a compact surface $S_{\gamma}$ equipped with a homeomorphism $h$ between these two boundary components of $S_{\gamma}$  so that the quotient $S_{\gamma}/(x \sim h(x))$ is homeomorphic to $S$ and the image of these distinguished boundary components under the quotient map is $\gamma$.

If $X_1$ and $X_2$ are topological spaces and $A_1 \subset X_1$, $A_2 \subset X_2$ so that $A_1 \cong A_2$, we say $X$ is obtained by {\it identifying} $X_1$ and $X_2$ along $A_1$ and $A_2$ if $X = X_1 \sqcup X_2 /(x\sim h(x))$ for some homeomorphism $h:A_1 \rightarrow A_2$ and all $x \in A_1$. If $A$ is the image of $A_1$ and $A_2$ under the quotient map, we denote the space $X$ as $X = X_1 \cup_{A} X_2$. 

Let $\mathcal{X}$ denote the class of spaces homeomorphic to two hyperbolic surfaces identified along an essential closed curve in each. Let $\mathcal{X}_S \subset \mathcal{X}$ be the subclass in which the curves that are identified are simple. Let $\mathcal{C}$ be the class of groups isomorphic to the fundamental group of a space in $\mathcal{X}$, and let $\mathcal{C}_S \subset \mathcal{C}$ be the subclass of groups isomorphic to the fundamental group of a space in $\mathcal{X}_S$. If $G \in \mathcal{C}$ then $G \cong \pi_1(S_{g})*_{\left\langle \gamma \right\rangle} \pi_1(S_{h})$, the amalgamated free product of two hyperbolic surface groups over $\Z$. We suppress in our notation the monomorphisms $i_g:\left\langle\gamma\right\rangle \rightarrow \pi_1(S_{g})$ and $i_h: \left\langle\gamma\right\rangle \rightarrow \pi_1(S_h)$ given by $i_g:\gamma \mapsto [\gamma_g]$, $i_h:\gamma \mapsto [\gamma_h]$, where $\gamma_g:S^1 \rightarrow S_g$ and $\gamma_h:S^1 \rightarrow S_h$. Note that if $X \in \mathcal{X}_S$ consists of two surfaces identified to each other along separating curves, $\pi_1(X)$ may be expressed as an amalgamated free product of surface groups in up to three ways.

\section{Abstract commensurability classes within $\mathcal{C}_S$}

There are many notions of commensurability in group theory and topology. The first step taken in our abstract commensurability classification is to translate this algebraic question into a topological one, as described in the following section.

\subsection{Finite covers and topological rigidity}

A description of the subgroup structure of an amalgamated free product is given in the following theorem of Scott and Wall.

\begin{thm} \label{subgroup} (\cite{scottwall}, Theorem 3.7) {\it If $G \cong A*_C B$ and if $H \leq G$, then $H$ is the fundamental group of a graph of groups, where the vertex groups are subgroups of conjugates of $A$ or $B$ and the edge groups are subgroups of conjugates of $C$. } \end{thm}

Any finite sheeted cover of the space $X = S_g \cup_{\gamma} S_h$, where $\gamma$ is the image of $\gamma_g:S^1 \rightarrow S_g$ and $\gamma_h:S^1 \rightarrow S_h$ under identification, consists of a set of surfaces which cover $S_g$ and a set of surfaces which cover $S_h$, identified along multicurves that are the preimages of $\gamma_g$ and $\gamma_h$. These covers are examples of {\it simple, thick, $2$-dimensional hyperbolic P-manifolds} (see \cite{lafont}, Definition 2.3.) The following {\it topological rigidity} theorem of Lafont allows us to address the abstract commensurability classification for members in $\mathcal{C}_S$ from a topological point of view. Corollary \ref{top} also follows from the proof of Proposition 3.1 in \cite{crisppaoluzzi}.

\begin{thm} (\cite{lafont}, Theorem 1.2) \label{toprigidity} {\it
 Let $X_1$ and $X_2$ be a pair of simple, thick, $2$-dimensional hyperbolic $P$-manifolds, and assume that $\phi:\pi_1(X_1) \rightarrow \pi_1(X_2)$ is an isomorphism. Then there exists a homeomorphism $\Phi:X_1 \rightarrow X_2$ that induces $\phi$ on the level of fundamental groups. }
\end{thm}

\begin{corollary} \label{top} {\it
 Let $G, G' \in \mathcal{C}_S$ with $G \cong \pi_1(X)$, $G' \cong \pi_1(X')$ and $X, X' \in \mathcal{X_S}$. Then $G$ and $G'$ are abstractly commensurable if and only if $X$ and $X'$ have homeomorphic finite-sheeted covering spaces. }
\end{corollary}

We will make repeated use of the following lemma. 

\begin{lemma} \label{eulerchar} {\it 
 If $X$ is a CW-complex and $X'$ is a degree $n$ cover of $X$, then $\chi(X') = n\chi(X)$, where $\chi$ denotes Euler characteristic.  }
\end{lemma}

\subsection{Statement of the classification and outline of the proof}

The abstract commensurability classification in the class $\mathcal{C}_S$ is given in terms of the ratio of the Euler characteristic of the surfaces identified and the {\it topological type} of the curves identified, which is defined as follows. An essential simple closed curve $\gamma$ on a surface $S$ is {\it non-separating} if $S\backslash \gamma$ is connected and is {\it separating} if $S \backslash \gamma$ consists of two connected surfaces, $S_{r,1}$ and $S_{s,1}$, of lower genus and a single boundary component. 

\begin{defn} \label{toptype}
 The {\it topological type} of an essential simple closed curve $\gamma:S^1 \rightarrow S$, denoted $t(\gamma)$, is equal to one if the curve is non-separating and equal to $\frac{\chi(S_{r,1})}{\chi(S_{s,1})}$ if the curve separates $S$ into subsurfaces $S_{r,1}$ and $S_{s,1}$ and $\chi(S_{r,1}) \leq \chi(S_{s,1})$. 
\end{defn}

{\bf Theorem \ref{classification}.} (Abstract commensurability classification within $\mathcal{C}_S$.)
 {\it  If $G_1, G_2 \in \mathcal{C}_S$, then $G_1$ and $G_2$ are abstractly commensurable if and only if, up to relabeling, $G_1 \cong \pi_1(S_{g_1})*_{\left\langle a_1 \right\rangle} \pi_1(S_{g_1'})$ and $G_2 \cong \pi_1(S_{g_2})*_{\left\langle a_2 \right\rangle} \pi_1(S_{g_2'})$, the amalgams are given by the monomorphisms $a_i \mapsto [\gamma_i] \in \pi_1(S_{g_i})$ and  $a_i \mapsto [\gamma_i'] \in \pi_1(S_{g_i'})$, and the following conditions hold. 

  \quad \quad (a) \, $\displaystyle \frac{\chi(S_{g_1})}{\chi(S_{g_1'})} = \frac{\chi(S_{g_2})}{\chi(S_{g_2'})}$, \quad \quad  \quad (b) \, $t(\gamma_1) = t(\gamma_2)$, \quad \quad \quad (c) \, $t(\gamma_1') = t(\gamma_2')$. 
}

One direction of the proof is constructive: if $G_1 \cong \pi_1(X_1)$ and $G_2 \cong \pi_1(X_2)$ satisfy the conditions of the theorem, we construct a common (regular) cover of the spaces $X_1$ and $X_2$. The other direction of the proof has three steps:
\begin{enumerate}
 \item Construct finite covers $p_i: Y_i \rightarrow X_i$ so that $Y_i$ consists of four surfaces each with two boundary components, one colored red and one colored blue; all red boundary components are identified and all blue boundary components are identified to form the connected space $Y_i$ with two singular curves; and, $\chi(Y_1) = \chi(Y_2)$. The existence of such covers is proven in Lemma \ref{existcovers}, and an example of these covers is given in Figure \ref{coverssamechi}. 
 \item Apply Proposition \ref{samechi}, which generalizes \cite[Theorem 5.3]{malone}, and proves that since $G_1$ and $G_2$ are abstractly commensurable, the finite covers $Y_1$ and $Y_2$ are homeomorphic. 
 \item Use the covering maps $p_1$ and $p_2$ to label the surfaces in $X_1$ and $X_2$ so that $G_1$ and $G_2$ are expressed as in the theorem and the conditions (a), (b), and (c) hold. 
\end{enumerate}

\subsection{Abstract commensurability classification} In this section we prove Theorem \ref{classification}, characterizing the abstract commensurability classes in $\mathcal{C}_S$. To prove the conditions in the theorem are necessary, the first step, denoted (1) above, is to take covers of spaces $X_1, X_2 \in \mathcal{X}_S$ with abstractly commensurable fundamental groups so that the covers of $X_1$ and $X_2$ have equal Euler characteristic. 

 \begin{figure}[t]
   \includegraphics[height=6.2cm]{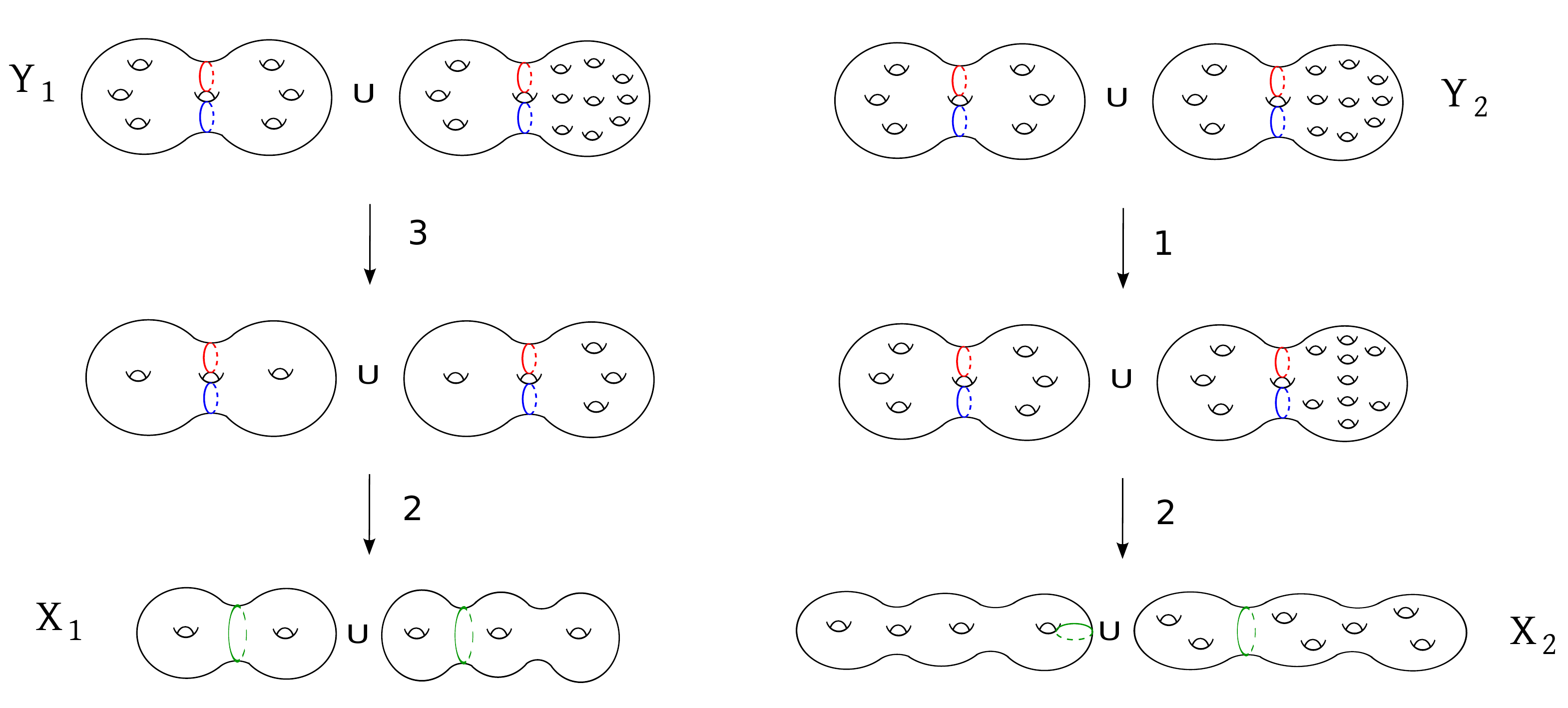}
   \caption[Finite covers of spaces in $\mathcal{X}_S$]{ {\small Above is an example of the covers $p_i:Y_i \rightarrow X_i$ constructed in Lemma \ref{existcovers}. In each union, the two curves of the same color are glued together to form singular curves. In this example, $\pi_1(X_1)$ and $\pi_1(X_2)$ are abstractly commensurable; one can check that conditions (a), (b), and (c) hold.  } }
  \label{coverssamechi}
 \end{figure}

\begin{lemma} \label{existcovers} 
 {\it  If $X_1, X_2 \in \mathcal{X}_S$, then there exist finite-sheeted covers $p_i:Y_i \rightarrow X_i$ so that $Y_i$ consists of four surfaces each with two boundary components, one colored red and one colored blue; all red boundary components are identified and all blue boundary components are identified to form the connected space $Y_i$ with two singular curves; and, $\chi(Y_1) = \chi(Y_2)$.  }
\end{lemma}

\begin{proof}
  Let $X_1, X_2 \in \mathcal{X}_S$. Let \[L = -2\cdot \ell cm(|\chi(X_1)|, |\chi(X_2)| )\] and \[d_i = \frac{L}{\chi(X_i) }.\] Suppose $X_1 = S_{h_1} \cup_{c_1} S_{h_1'}$ and $X_2 = S_{h_2} \cup_{c_2} S_{h_2'}$ where $c_i$ identifies the curves $\rho_i:S^1 \rightarrow S_{h_i}$ and $\rho_i':S^1 \rightarrow S_{h_i'}$. To build the covers $Y_i$, first let $\widetilde{S_{h_i}}$ be a $2$-fold cover of $S_{h_i}$ so that $\rho_i$ has two preimages in the cover: if $\rho_i$ is non-separating, cut along $\rho_i$, take two copies of the resulting surface with boundary, and re-glue the boundary components in pairs; if $\rho_i$ is separating, cut along a non-separating essential simple closed curve in each of the subsurfaces bounded by $\rho_i$, take two copies of the resulting surface with boundary, and re-glue the boundary components in pairs. An example of these degree two covers appears in Figure \ref{coverssamechi}. Next, cut along a non-separating curve in the cover $\widetilde{S_{h_i}}$ that intersects each curve in the pre-image of $\rho_i$ in exactly one point. Take $\frac{d_i}{2}$ copies of the resulting surface with two boundary components and reglue the boundary components in pairs to get a surface $\widehat{S_{h_i}}$ which forms a $\frac{d_i}{2}$-fold cyclic cover of $\widetilde{S_{h_i}}$ and so that $\rho_i$ has two preimages in $\widehat{S_{h_i}}$, each of which covers $\rho_i$ by degree $\frac{d_i}{2}$. Construct $\widehat{S_{h_i'}}$ in the same way. Identify the two components of the preimage of $\rho_i$ in $\widehat{S_{h_i}}$ with the two components of the preimage of $\rho_i'$ in $\widehat{S_{h_i'}}$ in pairs to form $Y_i$, a $d_i$-fold cover of $X_i$. An example of these covers is illustrated in Figure \ref{coverssamechi}. By construction, $\chi(Y_1) = \chi(Y_2) = L$. 
\end{proof}

We will apply the following proposition (with $r=4$ and $n=2$). The idea to restrict to the setting of spaces with equal Euler characteristic appears in \cite[Theorem 5.3]{malone}, though the proof there has a small gap in the inductive step. In our proof, below, we complete Malone's proof and generalize his result. 

\begin{prop} \label{samechi} {\it 
 Let $G_1 \cong \pi_1(X_1)$ and $G_2 \cong \pi_1(X_2)$ where 
 \begin{center}
  $\displaystyle X_1 = \bigcup_{i=1}^r S_i$ \quad and \quad $\displaystyle X_2 = \bigcup_{i=1}^r T_i$; 
 \end{center}
 $r\geq 3$; $S_i$ is a surface with $n$ boundary components $\{\beta_{i1}, \ldots, \beta_{in}\}$; boundary components $\beta_{ij}$ and $\beta_{kj}$ are identified for all $1 \leq j \leq n$ and $1 \leq i \leq k \leq r$ so there are $n$ singular curves in $X_1$; and $X_2$ is similar. Suppose that $\chi(S_1) \leq \ldots \leq \chi(S_r)$, $\chi(T_1) \leq \ldots \leq \chi(T_r)$, and $\chi(X_1) = \chi(X_2)$. Then $G_1$ and $G_2$ are abstractly commensurable if and only if $S_i \cong T_i$ for all $1 \leq i \leq r$.  }
\end{prop}

\begin{proof} Suppose $G_1$ and $G_2$ are abstractly commensurable. Then there exist finite covers $p_1:\hat{X}_1 \rightarrow X_1$ and $p_2:\hat{X}_2 \rightarrow X_2$ with $\pi_1(\hat{X}_1) \cong \pi_1(\hat{X}_2)$. Since $\chi(X_1) = \chi(X_2)$, the covering maps $p_1$ and $p_2$ have the same degree, $d$. By Theorem \ref{toprigidity}, there exists a homeomorphism $f:\hat{X}_1 \rightarrow \hat{X}_2$ inducing the isomorphism between $\pi_1(\hat{X}_1)$ and $\pi_1(\hat{X}_2)$. 

Suppose 
\begin{eqnarray}
 \chi(S_1) = \ldots = \chi(S_s) < \chi(S_{s+1}) \leq \ldots \leq \chi(S_r) \\
 \chi(T_1) = \ldots = \chi(T_t)< \chi(T_{t+1}) \leq \ldots \leq \chi(T_r)
\end{eqnarray}
for some $s,t \leq r$. Without loss of generality, $\chi(S_1) \leq \chi(T_1)$ and if $\chi(S_1) = \chi(T_1)$, then $s \geq t$. 

Consider the full preimage in $\hat{X}_1$ of the surfaces $S_1, \ldots, S_s$ of least Euler characteristic in $X_1$. Let \[\mathcal{A}_i = p_1^{-1}(S_i).\]
The surface $\mathcal{A}_i$ may be disconnected; suppose $\mathcal{A}_i$ is the disjoint union of $k_i$ connected surfaces, \[\mathcal{A}_i = \bigsqcup_{j=1}^{k_i}A_{ij} .\]

Each component $f(A_{ij})$ of $f(\mathcal{A}_i)$ covers some surface $T_{ij} \in \{T_1, \ldots, T_r \} \subset X_2$ under the covering map $p_2$. Suppose $p_2:f(A_{ij}) \rightarrow T_{ij}$ is a degree $d_{ij}$ cover. For each $i$, the sum of the degrees $d_{ij}$ is equal to $d$ since the boundary of $f(\mathcal{A}_i)$ is the full preimage of the $n$ singular curves in $X_2$ and no component of the preimage of the singular curves is incident to more than one component of $f(\mathcal{A}_i)$. Thus, 

\begin{eqnarray*}
 d \cdot \chi(S_1) &=& \sum_{j=1}^{k_1} \chi(A_{1j}) \\
 &=& \sum_{j=1}^{k_1} \chi(f(A_{1j})) \\
 &=& \sum_{j=1}^{k_1} d_{1j} \cdot\chi(T_{1j}) \\
 &\geq& \chi(T_1) \cdot \sum_{j=1}^{k_1} d_{1j} \\
 &=& d \cdot \chi(T_1)
\end{eqnarray*}
Since $\chi(S_1) \leq \chi(T_1)$ by assumption, $\chi(S_1) = \chi(T_1)$. Each singular curve in $\hat{X}_2$ is incident to $s$ surfaces in $f(\mathcal{A}_1) \cup \ldots \cup f(\mathcal{A}_s)$, so $p_2( f(\mathcal{A}_1) \cup \ldots \cup f(\mathcal{A}_s))$ must have in its image at least $s$ surfaces in $X_2$, each of which must have Euler characteristic equal to $\chi(S_1)$ by the above argument. Thus, since $s \leq t$, we have $\chi(S_i) = \chi(T_i)$ for $1 \leq i \leq s = t$. Moreover, $p_1^{-1}\left(\bigcup_{i=1}^s S_i)\right) = p_2^{-1}\left(\bigcup_{i=1}^s T_i)\right)$, so the above argument can be repeated (at most finitely many times) with the remaining surfaces in $X_1$ and $X_2$ of strictly larger Euler characteristic, proving the claim. 

The other direction of the statement is clear: if $a_i = b_i$ for $1 \leq i \leq r$, then $\pi_1(G_1) \cong \pi_1(G_2)$, so $G_1$ and $G_2$ are abstractly commensurable. 
\end{proof}

{\bf Remark:} The condition that $\chi(X_1) = \chi(X_2)$ can be omitted from the above proposition, and we get the conclusion that $\frac{\chi(S_i)}{\chi(T_i)} = c$ for some constant $c$ and all $1 \leq i \leq r$. This generalization appears in upcoming joint work with Pallavi Dani and Anne Thomas on abstract commensurability classes of certain right-angled Coxeter groups.

\begin{thm} \label{classification}  {\it  If $G_1, G_2 \in \mathcal{C}_S$, then $G_1$ and $G_2$ are abstractly commensurable if and only if they may be expressed as $G_1 \cong \pi_1(S_{g_1})*_{\left\langle a_1 \right\rangle} \pi_1(S_{g_1'})$ and $G_2 \cong \pi_1(S_{g_2})*_{\left\langle a_2 \right\rangle} \pi_1(S_{g_2'})$,  given by the monomorphisms $a_i \mapsto [\gamma_i] \in \pi_1(S_{g_i})$ and  $a_i \mapsto [\gamma_i'] \in \pi_1(S_{g_i'})$, and the following conditions hold. 

  \quad \quad (a) \, $\displaystyle \frac{\chi(S_{g_1})}{\chi(S_{g_1'})} = \frac{\chi(S_{g_2})}{\chi(S_{g_2'})}$, \quad \quad  \quad (b) \, $t(\gamma_1) = t(\gamma_2)$, \quad \quad \quad (c) \, $t(\gamma_1') = t(\gamma_2')$. 
}
\end{thm}

\begin{proof} Let $X_1, X_2 \in \mathcal{X}_S$. By Lemma \ref{existcovers}, there exist covering spaces $p_1: Y_1\rightarrow X_1$ and $p_2: Y_2 \rightarrow X_2$ so that $\chi(Y_1)=\chi(Y_2)$,
  \begin{center}
  $\displaystyle Y_1 = \bigcup_{i=1}^4 S_i$ \quad and \quad $\displaystyle Y_2 = \bigcup_{i=1}^4 T_i$; 
 \end{center} 
 the connected surfaces $S_i$ in $Y_1$ have two boundary components, one colored red and one colored blue; all red boundary components are identified and all blue boundary components are identified; and likewise for $Y_2$.
 
Suppose $G_1\cong \pi_1(X_1)$ and $G_2\cong\pi_1(X_2)$ are abstractly commensurable, so $\pi_1(Y_1)$ and $\pi_1(Y_2)$ are abstractly commensurable. By Proposition \ref{samechi}, $S_i \cong T_i$ for $1 \leq i \leq 4$. The conditions of the theorem require a labeling of the surfaces and amalgamated curves in $X_1$ and $X_2$. Thus, it remains to assign $S_{g_i}$, $S_{g_i'}$, $\gamma_i$, and $\gamma_i'$ for $i=1,2$ that satisfy conditions (a), (b), and (c). This assignment depends on whether the original curves $\rho_i$ and $\rho_i'$ are separating or non-separating. Let $p_1:Y_1 \rightarrow X_1$ and $p_2:Y_2 \rightarrow X_2$ be the covering maps constructed above.  

If the curves $\rho_i$ and $\rho_i'$ are separating for $i=1,2$, suppose $\chi(S_i) \leq \chi(S_j)$ for $i\leq j$. Let 
\begin{center}
 $S_{g_1} = p_1( S_1 ) \cup_{\gamma_1} p_1( S_2 )$ \, and \,  $S_{g_1'} = p_1( S_3 ) \cup_{\gamma_1'} p_1( S_4)$
\end{center}
 be the surfaces obtained by identifying $p_1(S_i)$ along their boundary curves and let $\gamma_i$ and $\gamma_i'$ be the images of the boundary curves. Similarly, let 
 \begin{center}
  $S_{g_2} = p_2( T_1) \cup_{\gamma_2} p_2( T_2 )$ \, and \, $S_{g_2'} = p_2( T_3 ) \cup_{\gamma_2'} p_2( T_4 )$. 
 \end{center}
 One can easily check that the conditions of the theorem hold:
 \begin{eqnarray*}
  t(\gamma_1) &=& \frac{\chi(p_1(S_1))}{\chi(p_1(S_2)) } \\
  &=& \frac{ \frac{ \chi(S_1) }{ d_1 }  }{  \frac{ \chi(S_2) }{ d_1 }  }\\ 
  &=& \frac{ \frac{ \chi(S_1) }{ d_2 }  }{  \frac{ \chi(S_2) }{ d_2 }  }\\ 
  &=&  \frac{ \frac{ \chi(T_1) }{ d_2 }  }{  \frac{ \chi(T_2) }{ d_2 }  }\\ 
  &=& \frac{\chi(p_2(T_1))}{\chi(p_2(T_2)) } \\
  &=& t(\gamma_2) ,
 \end{eqnarray*}

 and an analogous calculation shows $t(\gamma_1') = t(\gamma_2')$, proving claims (b) and (c). Similarly, 
 
 \begin{eqnarray*}
  \frac{ \chi(S_{g_1})}{ \chi(S_{g_1'} )} &=& \frac{ \chi(p_1(S_1 \cup S_2)) }{ \chi(p_1(S_3 \cup S_4)) } \\
  &=& \frac{ \frac{ \chi(S_1 \, \cup \, S_2) }{ d_1} }{ \frac{ \chi(S_3\,\cup\, S_4) }{ d_1} } \\
  &=& \frac{ \frac{ \chi(S_1 \, \cup \, S_2) }{ d_2} }{ \frac{ \chi(S_3\,\cup\, S_4) }{ d_2} } \\
   &=& \frac{ \frac{ \chi(T_1 \, \cup \, T_2) }{ d_2} }{ \frac{ \chi(T_3\,\cup\, T_4) }{ d_2} } \\
  &=&  \frac{ \chi(p_2(T_1 \cup T_2)) }{ \chi(p_2(T_3 \cup T_4)) } \\
  &=& \frac{ \chi(S_{g_2})}{ \chi(S_{g_2'} )}, 
 \end{eqnarray*}
 establishing (a) in this case. 
 
 Otherwise, at least one amalgamating curve $\rho_i$ or $\rho_i'$ is non-separating for $i = 1$ or $i= 2$. By the construction of the covers $p_i:Y_i \rightarrow X_i$, this situation implies $S_i \cong S_j$ for some $i \neq j$. Let $k$ and $\ell$ denote the other indices. There are now three cases: among the $S_i$ (and $T_i \cong S_i$) either two, three, or four of these connected surfaces with boundary are homeomorphic. 
 
 If neither $S_k$ nor $S_{\ell}$ is homeomorphic to $S_i$, define 
 \begin{center}
  $S_{g_1} = p_1(S_i) \cup_{\gamma_1} p_1(S_j)$, \\
  $S_{g_1'} = p_1(S_k) \cup_{\gamma_1'} p_1(S_{\ell})$, \\
  $S_{g_2} = p_2(T_i) \cup_{\gamma_2} p_2(T_j)$, \\
  $S_{g_2'} = p_2(T_k) \cup_{\gamma_2'} p_2(T_{\ell})$. \\
 \end{center}
 
 If, without loss of generality, $S_k \cong S_i$ and $S_{\ell} \neq S_i$, let $S_{g_1}$ and $S_{g_2}$ be the surfaces covered by two of $\{S_i, S_j, S_k \}$, and let $S_{g_1'}$ and $S_{g_2'}$ be covered by the remaining two subsurfaces. Let $\gamma_i$ and $\gamma_i'$ be the images of the boundary curves under the covering maps. Finally, if all four surfaces $S_i$ are homeomorphic, define $(S_{g_i}, \gamma_i) = (S_{h_i}, \rho_i)$ and $(S_{g_i'}, \gamma_i') = (S_{h_i'}, \rho_i')$ to be the spaces given by the original labeling. In all three cases, conditions (a), (b), and (c) are verified in a manner similar to that above. 
 
 Suppose now that $G_1$ and $G_2$ are expressed as in the statement of the theorem and that conditions (a), (b), and (c) hold. Let $X_1 = S_{g_1} \cup_{c_1} S_{g_1'}$ and $X_2 = S_{g_2} \cup_{c_2} S_{g_2'}$ be the corresponding spaces where $c_i$ identifies the essential simple closed curves $\gamma_i:S^1 \rightarrow S_{g_i}$ and $\gamma_i':S^1 \rightarrow S_{g_i'}$.  Construct finite covers $p_1:Y_1\rightarrow X_1$ of degree $d_1$ and $p_2:Y_2\rightarrow X_2$ of degree $d_2$ as in Lemma \ref{existcovers}, with $S_{g_i}$, $S_{g_i'}$, $\gamma_i$, and $\gamma_i'$ replacing $S_{h_i}$, $S_{h_i'}$, $\rho_i$, and $\rho_i'$, respectively. We claim that $Y_1$ and $Y_2$ are homeomorphic. Let 
\begin{center}
 $S_1 \cup S_2 = p_1^{-1}(S_{g_1})$, \\
 $S_3 \cup S_4 = p_1^{-1}(S_{g_1'})$, \\
 $T_1 \cup T_2 = p_2^{-1}(S_{g_2})$, \\ 
 $T_3 \cup T_4 = p_2^{-1}(S_{g_2'})$.\\
\end{center}
Suppose $\chi(S_1) \leq \chi(S_2)$, $\chi(S_3) \leq \chi(S_4)$, $\chi(T_1) \leq \chi(T_2)$, and $\chi(T_3) \leq \chi(T_4)$; we use the conditions of the theorem to show $S_i \cong T_i$ for $1 \leq i \leq 4$. Since 
\begin{center}
 $d_1\cdot \chi(S_{g_1}) = \chi(S_1 \cup S_2)$, \\
 $d_1\cdot \chi(S_{g_1'}) = \chi(S_3 \cup S_4)$, \\
 $d_2\cdot \chi(S_{g_2}) = \chi(T_1 \cup T_2)$, \\
 $d_2\cdot \chi(S_{g_2'}) = \chi(T_3 \cup T_4)$, 
\end{center}
by condition (a), 
\begin{eqnarray*}
\frac{ \chi(S_1 \cup S_2)}{\chi(S_3 \cup S_4) } &=& 
 \frac{\chi(S_{g_1})}{\chi(S_{g_1'})}  \\
 &=&   \frac{\chi(S_{g_2})}{\chi(S_{g_2'})} \\
 &=&  \frac{ \chi(T_1 \cup T_2)}{\chi(T_3\cup T_4) }.
\end{eqnarray*}
Since $\chi(Y_1) = \chi(Y_2) = L$, 
\begin{center}
 $\chi(S_1 \cup S_2) + \chi(S_3 \cup S_4) = \chi(T_1 \cup T_2) + \chi(T_3 \cup T_4),   $
\end{center}
hence
\begin{eqnarray}\label{sum}
 \chi(S_1 \cup S_2) = \chi (T_1 \cup T_2 ), 
\end{eqnarray}
\begin{eqnarray*}
 \chi(S_3 \cup S_4) = \chi (T_3 \cup T_4).
\end{eqnarray*}
By condition (b), $t(\gamma_1) = t(\gamma_2)$. If $t(\gamma_i) = 1$, then by construction $\chi(S_1) = \chi(S_2) = \chi(T_1) = \chi(T_2)$. Otherwise, 
\begin{eqnarray*}
\frac{\chi(S_1) }{ \chi(S_2) } &=& t(\gamma_1) \\
&=& t(\gamma_2) \\
&=& \frac{\chi(T_1) }{\chi(T_2) }, 
\end{eqnarray*}
so by equation (\ref{sum}) above (and since Euler characteristic sums over these unions), we have $\chi(S_i) = \chi(T_i)$ for $i=1,2$. By condition (c) and an analogous calculation, we conclude $\chi(S_i) = \chi(T_i) $ for all $1 \leq i \leq 4$. Thus, $Y_1 \cong Y_2$, and therefore $G_1$ and $G_2$ are abstractly commensurable. \end{proof}

\begin{corollary}
 {\it If $G_1, G_2 \in \mathcal{C}_S$ and $G_1$ and $G_2$ are abstractly commensurable, then there exist normal subgroups of finite index, $N_i \triangleleft G_i$ so that $N_1 \cong N_2$. }
\end{corollary}
\begin{proof}
 In the proof of Theorem \ref{classification}, the covers constructed are regular.
\end{proof}

\begin{figure}[t]
   \includegraphics[height=7.5cm]{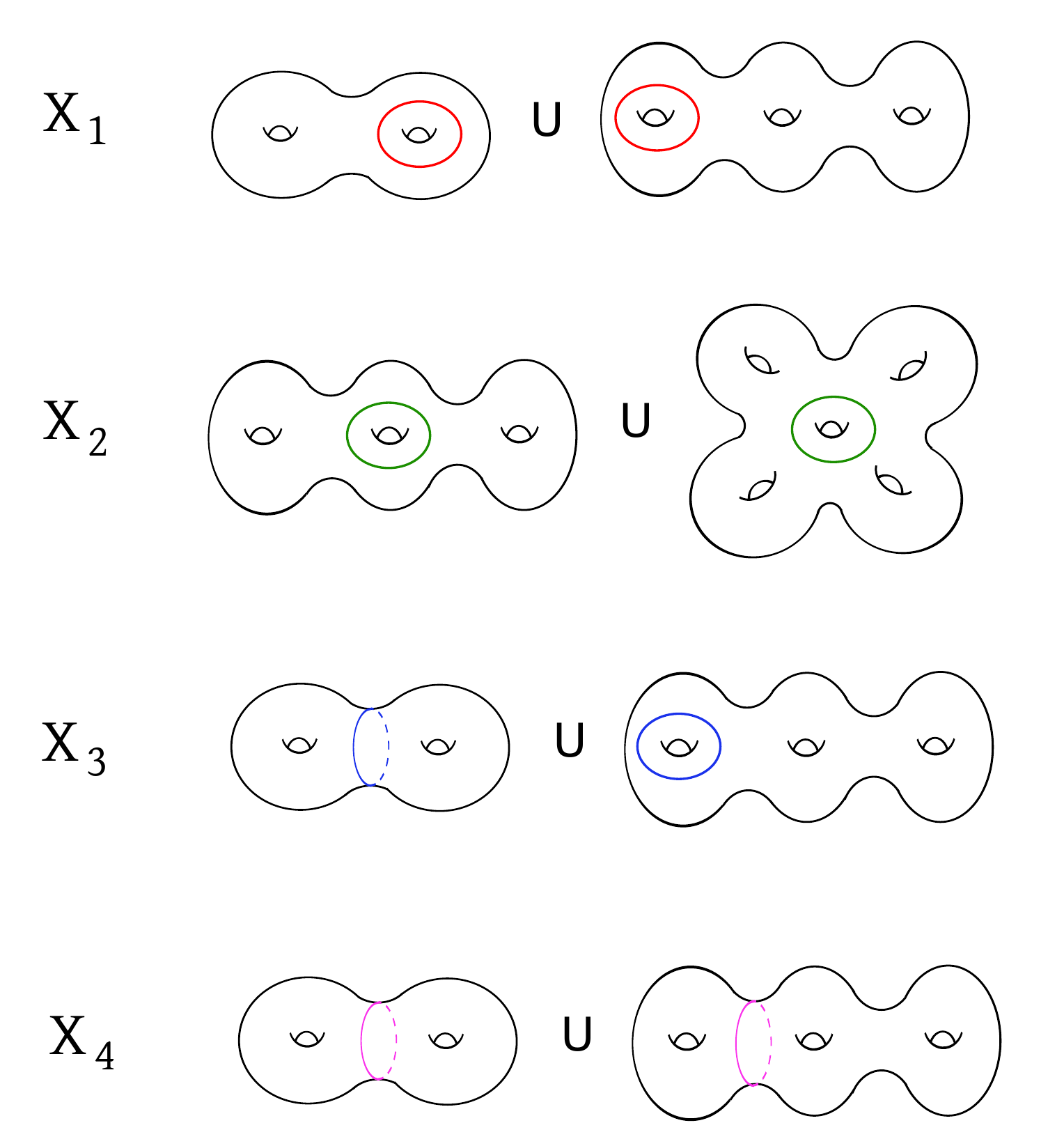}
   \caption[An example of spaces in $\mathcal{X}_S$ with abstractly commensurable fundamental groups ]{ {\small  {\it Example:} The groups $\pi_1(X_1)$, $\pi_1(X_2)$, and $\pi_1(X_3)$ are abstractly commensurable, but are not abstractly commensurable with $\pi_1(X_4)$. All four groups are quasi-isometric by Theorem \ref{bclassification}. } }
  \label{}
 \end{figure}

 In the case that $G_1$ and $G_2$ are the fundamental groups of surfaces glued along separating curves, we have the following.  
 
 \begin{corollary} \label{justsep}
  {\it If $S_1, S_2, S_3, S_4$ and $T_1, T_2, T_3, T_4$ are orientable surfaces of genus greater than or equal to one and with one boundary component, the $S_i$ are glued along their boundary to form $X_1$, and the $T_i$ are glued along their boundary to form $X_2$, then $\pi_1(X_1)$ and $\pi_1(X_2)$ are abstractly commensurable if and only if, up to reindexing, the quadruples $(\chi(S_1), \ldots, \chi(S_4))$ and $(\chi(T_1), \ldots, \chi(T_4))$ are equal up to integer scale. 
  }
 \end{corollary}

  If $G_1$ and $G_2$ are the fundamental groups of surfaces glued along non-separating curves, we have the following.  
 
 \begin{corollary} \label{justnonsep}
  {\it If $S_{g_i}$ and $S_{g_i'}$ are orientable surfaces of genus greater than one identified to each other along a non-separating curve in each to form the space $X_i$ for $i=1,2$, then $\pi_1(X_1)$ and $\pi_1(X_2)$ are abstractly commensurable if and only if, up to reindexing, $\displaystyle \frac{\chi(S_{g_1})}{\chi(S_{g_1'})} = \frac{\chi(S_{g_2})}{\chi(S_{g_2'})}$.
  }
 \end{corollary}

\section{Quasi-isometry classification within $\mathcal{C}_S$}

Let $G$ be a group in the class $\mathcal{C}_S$ so that $G \cong \pi_1(X)$, where $X$ is a space in the class $\mathcal{X}_S$. Suppose $X=S_g \cup_{\gamma} S_h $ where $S_g$ and $S_h$ are closed orientable surfaces of negative Euler characteristic and $\gamma$ denotes the image of the essential simple closed curves $\gamma_g:S^1 \rightarrow S_g$ identified to $\gamma_h:S^1 \rightarrow S_h$ in $X$. There are many metrics on $X$ through which the geometry of the group $G$ may be studied. 

\subsection{A CAT$(-1)$ metric on $\widetilde{X}$}

Let $M_{\kappa}^n$ denote the complete, simply connected, Riemannian $n$-manifold of constant sectional curvature $\kappa \in \R$. As described in \cite[Chapter I.2]{bridsonhaefliger}, depending on whether $\kappa$ is positive, negative, or zero, $M_{\kappa}^n$ can be obtained from one of $\mathbb{S}^n$, $\Hy^n$, or $\mathbb{E}^n$, respectively, by scaling the metric. 

\begin{defn} [see Chapter II.1 of \cite{bridsonhaefliger}]
 Let $\Delta(p,q,r)$ be a geodesic triangle in a metric space $X$, which consists of three vertices $p$, $q$, and $r$, and three geodesic segments $[p,q]$, $[q,r]$, and $[r,p]$. A triangle $\bar{\Delta}(\bar{p}, \bar{q}, \bar{r}) \subset M_{\kappa}^2$ is called a {\it comparison triangle} for $\Delta(p,q,r)$ if $d(\bar{p}, \bar{q}) = d(p,q)$, $d(\bar{q}, \bar{r}) = d(q,r)$, and $d(\bar{r}, \bar{p}) = d(r,p)$.  A point $\bar{x} \in [\bar{q}, \bar{r}]$ is called a {\it comparison point} for $x \in [q,r]$ if $d(q,x) = d(\bar{q}, \bar{x})$. 
\end{defn}

\begin{defn} [see Definition II.1.1 of \cite{bridsonhaefliger}]  Let $X$ be a metric space and let $\kappa \in \R$. Let $\Delta$ be a geodesic triangle in $X$ with perimeter less than twice the diameter of $M_{\kappa}^2$. Let $\bar{\Delta} \subset M_{\kappa}^2$ be a comparison triangle for $\Delta$. Then $\Delta$ satisfies the {\it CAT$(\kappa)$ inequality} if for all $x,y \in \Delta$ and comparison points $\bar{x}, \bar{y} \in \bar{\Delta}$, $d(x,y) \leq d(\bar{x}, \bar{y})$. If $\kappa \leq 0$, then $X$ is called a {\it CAT$(\kappa)$ space} if $X$ is a geodesic space all of whose triangles satisfy the CAT$(\kappa)$ inequality. 
\end{defn}

In \cite{malone}, Malone proves all groups in $\mathcal{C}_S$ are quasi-isometric by examining a CAT$(0)$ geometry on $X$ and applying the techniques of Behrstock--Neumann on the bilipschitz equivalence of fattened trees \cite{behrstockneumann}. The bilipschitz equivalence constructed by Behrstock--Neumann relies on the Euclidean structure of fattened trees; their map is piecewise-linear.  In this paper, we study a CAT$(-1)$ metric on $X$ that is piecewise hyperbolic, and we define a bilipschitz equivalence with respect to this hyperbolic structure. The piecewise hyperbolic metric on $X \in \mathcal{X}_S$ can be constructed as follows. 





 One can choose hyperbolic metrics on $S_g$ and $S_h$ so that the length of the geodesic representatives of $[\gamma_g]$ and $[\gamma_h]$ is equal (see Chapter 10 of \cite{farbmargalit}). Gluing by an isometry yields a piecewise hyperbolic complex $X$. We call such a metric {\it hyperbolic on each surface.} The universal cover $\widetilde{X}$ consists of copies of $\Hy^2$ that are the lifts of the hyperbolic surfaces, identified along geodesic lines that are the lifts of the curve $\gamma$. The following proposition implies that $\widetilde{X}$ is a CAT$(-1)$ metric space. 
 
\begin{prop} \cite[Proposition II.11.6]{bridsonhaefliger}
 {\it Let $X_1$ and $X_2$ be metric spaces of curvature $\leq \kappa$ and let $A_1 \subset X_1$ and $A_2 \subset X_2$ be closed subspaces that are locally convex and complete. If $j:A_1 \rightarrow A_2$ is a bijective local isometry, then the quotient of the disjoint union $X = X_1 \bigsqcup X_2$ by the equivalence relation generated by $[a_1 \sim j(a_1)$ for all $a_1 \in A_1]$ has curvature $\leq \kappa$.  }
\end{prop}


For details on metric gluing constructions, see the work of Bridson--Haefliger (\cite{bridsonhaefliger}, Section II.11).
 

\subsection{Bilipschitz maps and polygonal tilings}

The bilipschitz equivalence between the universal covers of two spaces $X_1$ and $X_2$ in $\mathcal{X}_S$ is constructed by realizing $\widetilde{X}_1$ and $\widetilde{X}_2$ as isomorphic cell complexes with finitely many isometry types of hyperbolic polygons as cells. We will use the following definitions.

\begin{defn}
  A map $f:(X,d_X) \rightarrow (Y,d_Y)$ is {\it $K$-bilipschitz} if there exists $K \geq 1$ so that for all $x_1, x_2 \in X$, \[\frac{1}{K}d_X(x_1,x_2) \leq d_Y(f(x_1), f(x_2)) \leq Kd_X(x_1,x_2),\] and $f$ is a {\it $K$-bilipschitz equivalence} if, in addition, $f$ is a homeomorphism. A map is said to be a {\it bilipschitz equivalence} if it is a $K$-bilipschitz equivalence for some $K$. Two spaces $X$ and $Y$ are {\it bilipschitz equivalent} if there exists a bilipschitz equivalence from $X$ to $Y$.
\end{defn}

\begin{example}
  The map $f:[0,D] \rightarrow [0,D']$ given by $x \mapsto \frac{D'}{D}x$ is called {\it dilation}, and is a bilipschitz equivalence with bilipschitz constant $\frac{D'}{D}$. 
\end{example}

\begin{defn} 
  A {\it convex hyperbolic polygon} is the convex hull of a finite set of points in the hyperbolic plane. 
\end{defn}

\begin{lemma} \label{triangles} 
  {\it Let $\Delta_1, \Delta_2 \subset \Hy^2$ be hyperbolic triangles. Then there exists a bilipschitz equivalence $\phi:\Delta_1 \rightarrow \Delta_2$ that is dilation when restricted to each edge of $\Delta_1$.  } 
\end{lemma}

\begin{proof}
 It follows from \cite[Lemma 5, Lemma 6]{belenkiiburago} that there is a bilipschitz equivalence between a hyperbolic triangle and its Euclidean comparison triangle that restricts to an isometry on each of the edges. Then, composing with a linear map between Euclidean triangles gives the desired result. 
\end{proof}

\begin{corollary} \label{cells}
  {\it If $P$ and $Q$ are convex hyperbolic $n$-gons, then there exists a bilipschitz equivalence $\phi:P \rightarrow Q$ that is dilation when restricted to each edge of $P$.}
\end{corollary}


For a more formal and general definition of polyhedral complexes and their metric, see \cite[Chapter 1.7]{bridsonhaefliger}. 

\begin{lemma} \label{tilings} {\it If $\widetilde{X}_1$ and $\widetilde{X}_2$ are geodesic metric spaces realized as isomorphic cell complexes with finitely many isometry types of hyperbolic polygons as cells, then $\widetilde{X}_1$ and $\widetilde{X}_2$ are bilipschitz equivalent. }
\end{lemma}

\begin{proof}
Suppose geodesic metric spaces $\widetilde{X}_1$ and $\widetilde{X}_2$ are realized as isomorphic cell complexes with polygonal cells $\{V_i\}_{i \in I}$ and $\{W_i\}_{i \in I}$, respectively. Suppose the cell complex isomorphism maps $V_i$ to $W_i$ for all $i \in I$. By Corollary \ref{cells} and since there are finitely many isometry types of hyperbolic polygons in the cell complexes, we may take this map $\phi_i:V_i \rightarrow W_i$ to be a $K$-bilipschitz equivalence for some $K \in \R$ that restricts to dilation on each of the edges of $V_i$. These maps agree along the intersection of two polygons, thus, there is a well-defined cellular homeomorphism $\Phi: \widetilde{X}_1 \rightarrow \widetilde{X}_2$ that restricts to the $K$-bilipschitz equivalence $\phi_i$ on each cell.

Let $x,y \in \widetilde{X}_1$, and let $p$ be the geodesic path from $x$ to $y$. Since the cell complex contains finitely many isometry types of convex hyperbolic polygons, the path $p$ can be decomposed into a finite union of geodesic segments $\{[x_i, x_{i+1}]\}_{i=0}^{n-1}$, with $x_0=x$ and $x_n=y$, and so that each subpath $[x_i, x_{i+1}]$ is contained entirely in a $2$-cell $V_i$. Since $\Phi(p)$ is a path connecting $\Phi(x)$ and $\Phi(y)$, 
\begin{eqnarray*}
 d(\Phi(x), \Phi(y)) &\leq& \sum_{i=0}^{n-1} d(\phi_i(x_i),\phi_i(x_{i+1}))\\
&\leq& \sum_{i=0}^{n-1} Kd(x_i, x_{i+1})\\
&=& Kd(x,y).
\end{eqnarray*}
The other inequality follows similarly. Namely, suppose $q$ is a geodesic path from $\Phi(x)$ to $\Phi(y)$. The path $q$ can be decomposed into a union of geodesic segments $\{[w_i,w_{i+1}]_{i=0}^{m-1}\}$ where $w_0 = \Phi(x)$, $w_m = \Phi(y)$ and the interior of $[w_i, w_{i+1}]$ is contained entirely in a $2$-cell $W_i$. Then, since $\Phi^{-1}(q)$ is a path from $x$ to $y$ and $\phi_i$ is a $K$-bilipschitz equivalence for all $i$,
\begin{eqnarray*}
 d(\Phi(x),\Phi(y)) &=& \sum_{i=0}^{m-1} d(w_i, w_{i+1})\\
&\geq& \sum_{i=0}^{m-1} \frac{1}{K}d(\phi_i^{-1}(w_i), \phi_i^{-1}(w_{i+1}))\\
&\geq& \frac{1}{K}d(x,y).
\end{eqnarray*}
Thus, $\frac{1}{K}d(x,y) \leq d(\Phi(x), \Phi(y)) \leq Kd(x,y),$ so $\Phi$ is a $K$-bilipschitz equivalence.
\end{proof}

In the construction of the bilipschitz equivalence, we find it useful to restrict to a specific metric on a space $X \in \mathcal{X}_S$, and we will use the following lemma. 

\begin{lemma} \label{ACbilip}
 {\it If $X_1, X_2 \in \mathcal{X}_S$ and $\pi_1(X_1)$ and $\pi_1(X_2)$ are abstractly commensurable, then $\widetilde{X}_1$ and $\widetilde{X}_2$ are bilipschitz equivalent with respect to any CAT$(-1)$ metric on $X_1$ and $X_2$ that is hyperbolic on each surface. }
\end{lemma}

\begin{proof}
 Let $X_1, X_2 \in \mathcal{X}_S$, and suppose $\pi_1(X_1)$ and $\pi_1(X_2)$ are abstractly commensurable. By Theorem \ref{toprigidity}, there exist finite-sheeted covers $Y_i \rightarrow X_i$ that are homeomorphic. Choose a locally CAT$(-1)$ metric on $X_1$ and $X_2$ that is hyperbolic on each surface. This piecewise hyperbolic metric on $X_i$ lifts to a piecewise hyperbolic metric on $Y_i$. Since $Y_1$ and $Y_2$ are homeomorphic, we may realize $Y_1$ and $Y_2$ as finite simplicial complexes with isomorphic $1$-skeleta. After subdividing if necessary, we may assume each triangle in $Y_i$ is isometric to a hyperbolic triangle. So, $\widetilde{Y}_1 \equiv \widetilde{X}_1$ and $\widetilde{Y}_2 \equiv \widetilde{X}_2$ may be realized as simplicial complexes with isomorphic $1$-skeleta and each built from finitely many isometry types of hyperbolic triangles. By Lemma \ref{tilings},  $\widetilde{Y}_1 \equiv \widetilde{X}_1$ and $\widetilde{Y}_2 \equiv \widetilde{X}_2$ are bilipschitz equivalent. 
 \end{proof}
 


\vskip.2in

\subsection{Construction of the cellular isomorphism}

\begin{thm} \label{bclassification} 
 {\it If $X_1, X_2 \in \mathcal{X}_S$ and $\widetilde{X}_1$ and $\widetilde{X}_2$ are their universal covers equipped with a CAT$(-1)$ metric that is hyperbolic on each surface, then there exists a bilipschitz equivalence $\widetilde{X}_1 \rightarrow \widetilde{X}_2$.  }
\end{thm}

\begin{proof}
   Let $X_1, X_2, \in \mathcal{X}_S$. If $X \in \mathcal{X}_S$, then by the abstract commensurability classification within $\mathcal{C}_S$ given in Theorem \ref{classification}, there exists $Y \in \mathcal{X}_S$ so that $Y$ consists of four surfaces of genus at least two and one boundary component, identified to each other along their boundary components and so that $\pi_1(X)$ and $\pi_1(Y)$ are abstractly commensurable. So, by Lemma \ref{ACbilip}, it suffices to consider the case where
   \[ X_1 = \bigcup_{i=1}^4 S_i, \] \[ X_2 = \bigcup_{i=1}^4 T_i  ,\] where $S_i$ is a surface of genus greater than two and one boundary component for $1 \leq i \leq 4$, and the union identifies the boundary components of the $S_i$; the space $X_2$ is similar. Choose locally CAT$(-1)$ metrics on $X_1$ and $X_2$ that are hyperbolic on each surface, and let $\widetilde{X}_i$ denote the universal cover of $X_i$ equipped with this metric.
 
  Let $\gamma_i$ denote the singular curve in $X_i$ and let $\tg_i$ represent the component of the preimage of $\gamma_i$ in $\widetilde{X}_i$ stabilized by $\la [\gamma_i] \ra$. Let $\mathcal{L}_i = \{ g \cdot \tg_i \, | \, g \in \pi_1(X_i)\}$. Let $H_1, H_2, H_3, H_4$ be the four components of $\widetilde{X}_1 \backslash \mathcal{L}_1$ incident to $\tg_1$ so that $\pi_1(S_i)$ stabilizes $H_i$, and let $J_1, J_2, J_3, J_4$ be the four components of $\widetilde{X}_2 \backslash \mathcal{L}_2$ incident to $\tg_2$ so that $\pi_1(T_i)$ stabilizes $J_i$. 
 
 
    \begin{figure}[t]
   \includegraphics[height=9.0cm]{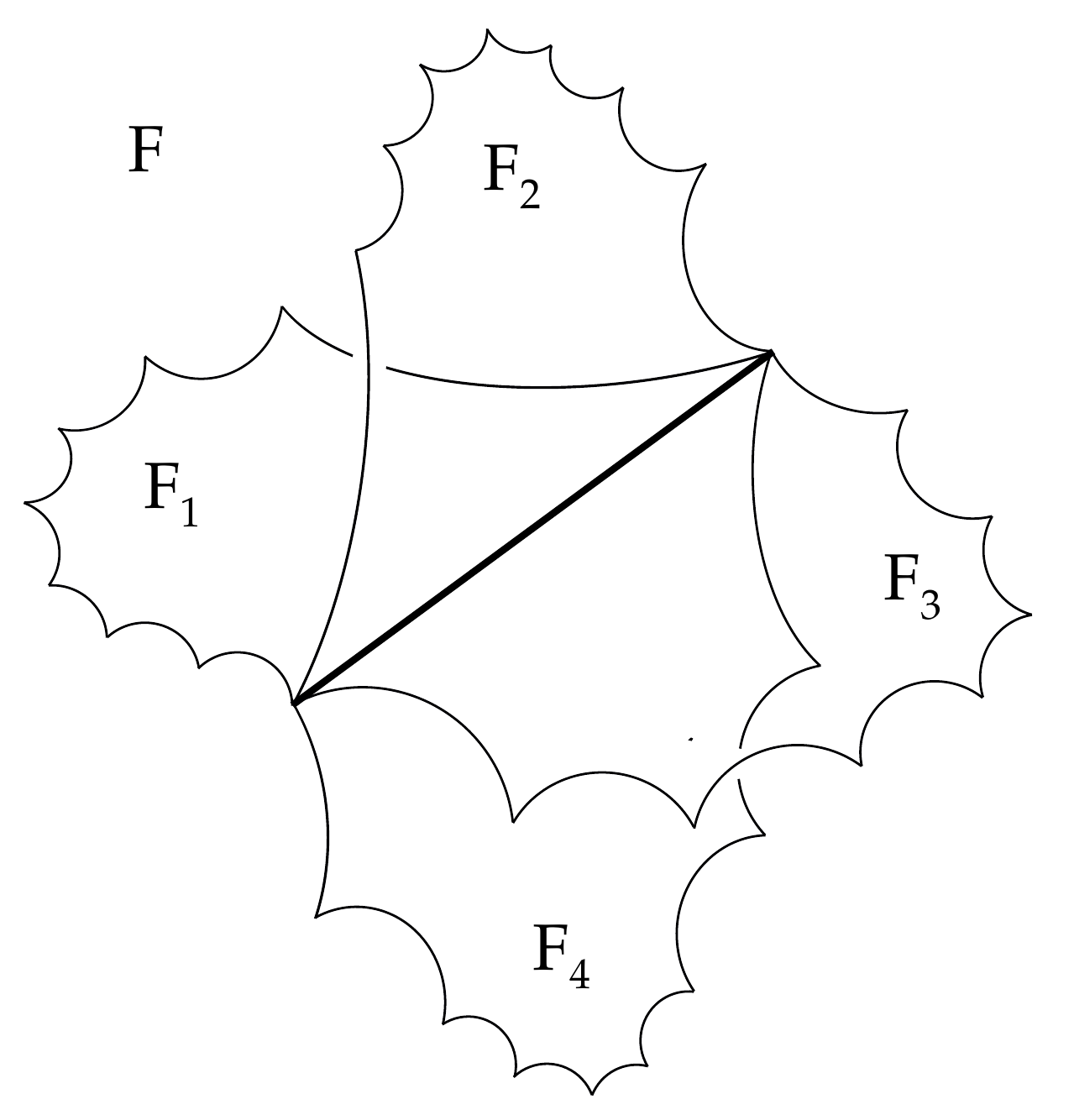}
   \caption[A fundamental domain]{ {\small An illustration of the fundamental domain $F$ for the action of $\pi_1(X_1)$ on $\widetilde{X}_1$. The fundamental domain is built from four convex hyperbolic polygons $F_i$. The darkened edge is referred to as the {\it branching edge}.} }
  \label{fdsname}
\end{figure}
 
  Let $F = \displaystyle \bigcup_{i=1}^4 F_i$ be a connected fundamental domain for the action of $\pi_1(X_1)$ on $\widetilde{X}_1$ that comes from a cell division of $X_1$ with a single vertex and so that 
  \begin{itemize}
   \item $F_i \subset H_i$ is a fundamental domain for the action of $\pi_1(S_i)$ on $H_i$, 
   \item $F_i$ is a convex hyperbolic polygon with at least nine sides so that exactly one edge of $F_i$ lies in $\tg_1$. We refer to this distinguished edge as the {\it branching edge} of $F_i$. The remaining vertices of $F_i$ lie on $g\tg_1$ for distinct $g \in \pi_1(X_1)$, 
   \item the branching edges $F_i$ are identified via an isometry to form the connected fundamental domain $F$. 
  \end{itemize}
An example is given in Figure \ref{fdsname}. Let $D = \displaystyle \bigcup_{i=1}^4 D_i$ be a connected fundamental domain for the action of $\pi_1(X_2)$ on $\widetilde{X}_2$ constructed similarly. Note that $F$ and $D$ are not {\it strict fundamental domains} (see \cite[Definition II.12.7]{bridsonhaefliger}); in particular, $F$ and $D$ contain many vertices.

 \begin{center}
  {\it Isometry types of cells used in the cell decompositions: }
 \end{center}
 
 Let $x$ and $y$ be one endpoint of the branching edges in $F$ and $D$, respectively. We will show that each polygon in the cell complexes constructed lies in the finite set of polygons $\mathcal{P}$ that satisfy the following three conditions. 
 
 \begin{itemize}
  \item The vertex sets are 
   \begin{center}
     $\mathcal{V}_1 = \{g \cdot x \, | \, g \in \pi_1(X_1)\}$ \quad and \quad $\mathcal{V}_2 = \{g \cdot y \, | \, g \in \pi_1(X_2) \}$,
   \end{center}
    respectively, the same vertices that appear in the tilings by fundamental domains.
  \item Each edge is isometric to a geodesic segment connecting two vertices of $F$ or $D$. 
  \item The number of sides of each polygon is bounded above by $M \in \N$, where $M$ is two times the maximum number of sides in $F$ or $D$ times the maximum valance $x$ or $y$. 
 \end{itemize}
  
 \begin{center}
   {\it Construction of the first cell in $H_1$ and $J_1$:}
 \end{center}
 
 \begin{figure}[t]
   \includegraphics[height=10.0cm]{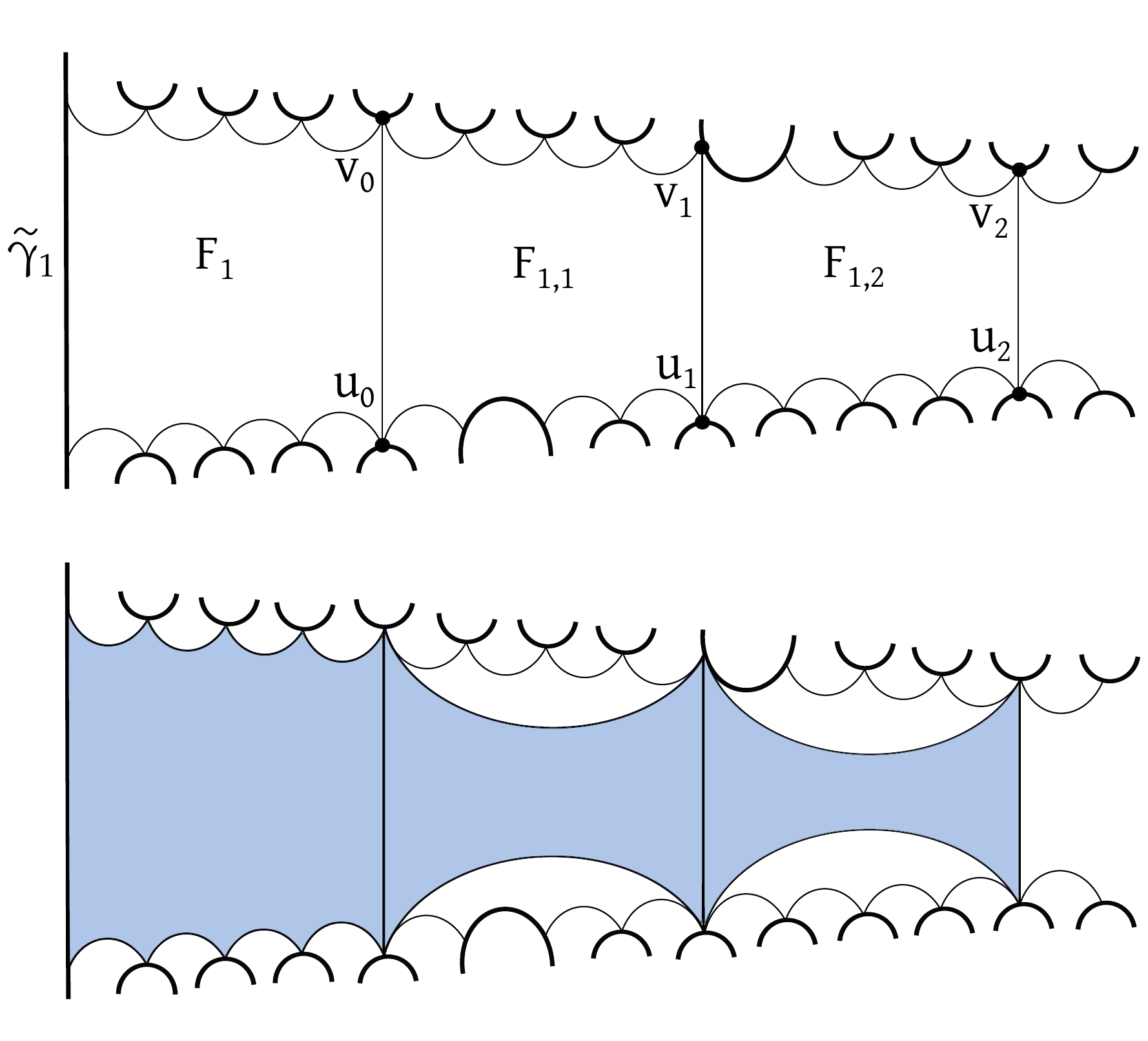}
   \caption[Translates of the fundamental domain and the first cell in the cell division of $\widetilde{X}$ for $X \in \mathcal{X}_S$]{ {\small In the top image are translates $F_{1,i}$ of the fundamental domain $F_1$ in $H_1$. The dark lines are translates of $\tg_1$, which are boundary lines of the region $H_1$. The vertices $u_i$ and $v_i$ are selected as in the proof of the theorem. Shaded below is the first tile in $H_1$ in the setting where the fundamental domain $D_1$ has more sides than the fundamental domain $F_1$.} }
  \label{maketile1}
\end{figure}
 
 Let $V$ be the vertices in the fundamental domain $F_1$ and let $W$ be the vertices in the fundamental domain $D_1$. If $V$ and $W$ have the same size, the fundamental domains themselves are the first cells used in the cell decomposition of $H_1$ and $J_1$; continue to the definition of the map. Otherwise, without loss of generality, $|W|-|V| = k >0$. We will enlarge $V$ until $|V| = |W|$. 
 
 Suppose $ k = 2n+m$ for some $n \geq 0$ and $m  \in \{0,1\}$. By the choice of the fundamental domain, there is a non-branching edge $\{u_0, v_0\}$ of $F_1$ that is disjoint from the branching edge of $F_1$ and its two adjacent edges. The edge $\{u_0, v_0\}$ lies in a second translate of the fundamental domain $F_{1,1} \subset H_1$. There is a non-branching edge $\{u_1, v_1\}$ in $F_{1,1}$ disjoint from $\{u_0, v_0\}$ and its two adjacent edges. Similarly, there are edges $\{u_i, v_i\}$ for $1 \leq i \leq n+1$, where $\{u_i, v_i\}$ and $\{u_{i+1}, v_{i+1}\}$ lie in the same fundamental domain $F_{1,i+1}$ for $1 \leq i \leq n$, and $\{u_i, v_i\}$ is disjoint from $\{u_j, v_j\}$ and its two adjacent edges for $i \neq j$, as illustrated in Figure \ref{maketile1}. 
 
 To construct the cycle boundary of $P_V$, the first cell in $H_1$, start with the cycle boundary of $F_1$. Remove the edge $\{u_0, v_0\}$. Add geodesic segments $\{u_i, u_{i+1}\}$ and $\{v_i, v_{i+1}\}$ for $1 \leq i \leq n-1$. Up to relabeling the $u_i$ and $v_i$, we may assume $\{u_i, u_{i+1}\}$ and $\{v_i, v_{i+1}\}$ do not intersect. If $k$ is even, add $\{u_n, v_n\}$ to complete the cycle boundary of the polygon. If $k$ is odd, add $\{u_n, u_{n+1}\}$ and $\{v_n, u_{n+1}\}$ to complete the cycle. Attach a $2$-cell to this boundary cycle to form the first cell $P_V$ in $H_1$. Let $P_W$ be the fundamental domain $D_1$, the first cell in $J_1$. 
 
 Map $P_V$ to $P_W$ by a cellular homeomorphism $\phi$, sending the branching edge of $P_V$ to the branching edge of $P_W$, and dilating along each edge of the tile. After extending the fundamental domain $F_1$ to the tile $P_V$, it is possible that $P_V$ is not convex. If this is the case, subdivide $P_V$ and $P_W$ isomorphically into convex polygons so the configurations have isomorphic $1$-skeleta. Observe that the number of edges in any polygon is bounded above by the size of the largest fundamental domain, and each edge connects vertices that lie in a common translate of the fundamental domain. Thus, $P_V, P_W \in \mathcal{P}$.

 \begin{center}
 {\it Constructing the remaining cells in $H_1$ and $J_1$: }
 \end{center}

 Extend the cell decompositions to all of $H_1$ and $J_1$ recursively. Along each new edge of a polygon built during the preceding stage, build one new polygon in $H_1$ and a corresponding new polygon in $J_1$. Each new polygon is constructed in a manner similar to the first polygons. Begin by constructing one new polygon along each edge of $P_V$ and $P_W$ that lies in the interior of $H_1$ and $J_1$ as follows. 
 
 Let $\{a, a_0\}$ be an edge of $P_V$ that lies in the interior of $H_1$ and let $\{b, b_0\} = \phi(\{a, a_0\})$. By construction, the edge $\{a, a_0\}$ connects two vertices in a translate of the fundamental domain, and the interior of this geodesic segment either lies on a non-branching edge of a translate of the fundamental domain or in the interior of a translate of the fundamental domain. This distinction does not affect the construction of the new cells.  The vertices $a$ and $a_0$ lie in distinct translates of $\tg_1$ that are boundary lines of $H_1$. Let $\{a, a'\}$ and  $\{a_0, a_0'\}$ be the branching edges on these translate of $\tg_1$ that lie in the component of $H_1 \backslash \{a, a_0\}$ that does not contain $P_V$. Let $\{b,b'\}$ and $\{b_0, b_0'\}$ be the analogous edges in $J_1$. We form cycles $C_A$ in $H_1$ and $C_B$ in $J_1$ that contain the paths $\{a',a,a_0,a_0'\}$ and $\{b',b,b_0, b_0' \}$, respectively, and will serve as the boundary cycles of the new cells constructed. The branching edges of the tiling by fundamental domains are distinguished; so, to ensure $C_A$ can be mapped to $C_B$, we extend these paths  $\{a',a,a_0,a_0'\}$ and $\{b',b,b_0, b_0' \}$ to cycles that contain no other branching edges.

   \begin{figure}[t]
   \includegraphics[height=9.0cm]{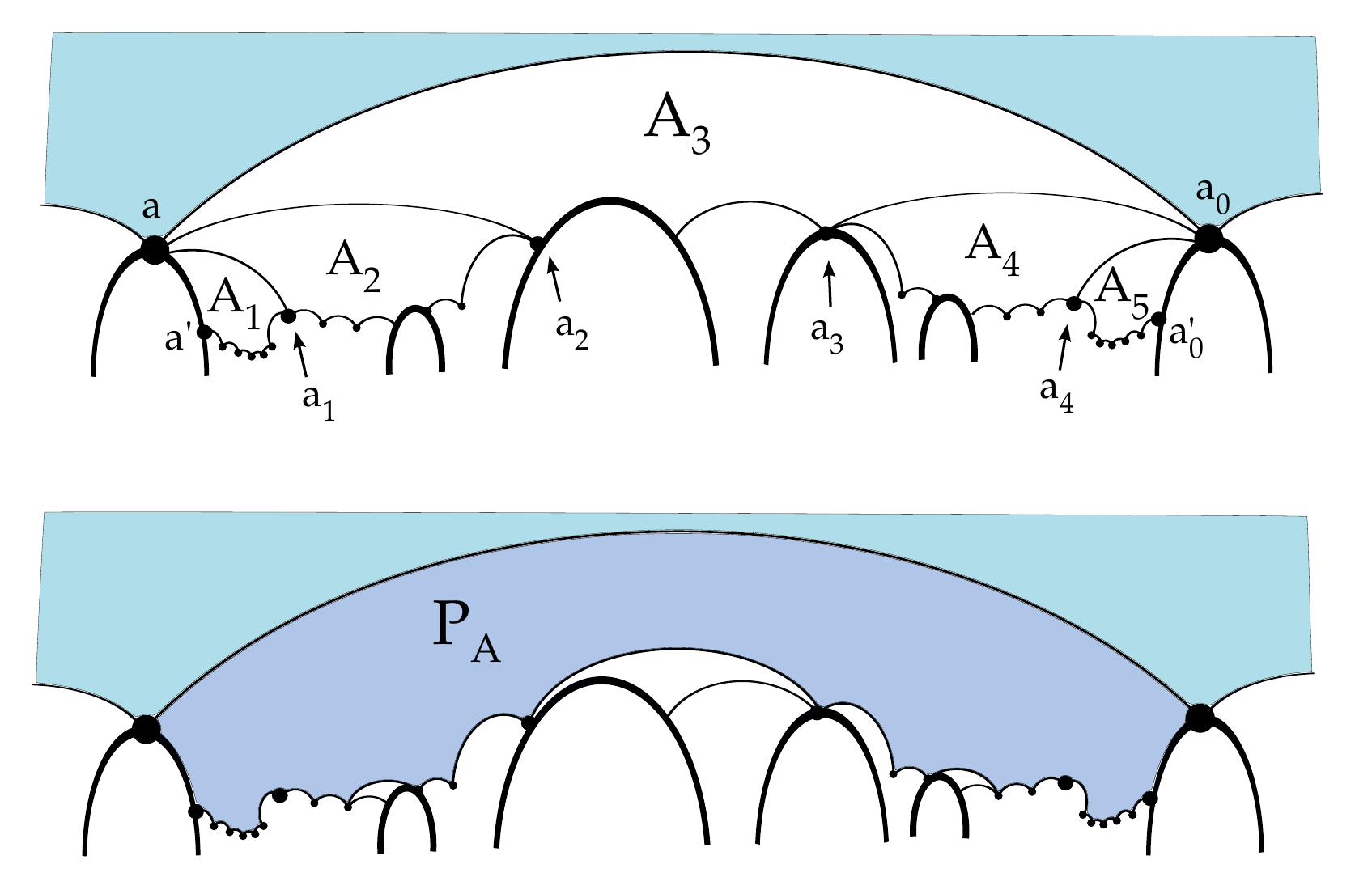}
   \caption[Translates of the fundamental domain and the second cell in the cell division of $\widetilde{X}$ for $X \in \mathcal{X}_S$]{ {\small The figures illustrate how to extend the tiling recursively along an edge $\{a,a_0\}$ of a (shaded) tile previously constructed. The $A_i$ are translates of the fundamental domain that intersect $a$ or $a_0$, and the $a_i$ are their points of intersection. The dark lines are translates of $\tg_1$ that bound $H_1$. Below, the new tile $P_A$ is drawn; its cycle boundary contains all of the $a_i$ as well as paths $p_i \subset A_i$ that include the other vertices of $A_i$, except that only one vertex is chosen from a boundary geodesic. Then, the only edges of $P_A$ that lie on the boundary geodesics are $\{a,a'\}$ and $\{a_0, a_0'\}$.  } }
  \label{maketile2}
\end{figure}
 
 Let $A_1, \ldots, A_m$ be the (non-empty) set of translates of the fundamental domain $F_1$ in $H_1$ that intersect $a$ or $a_0$ and the component of $H_1 \backslash \{a, a_0\}$ that does not contain $P_V$. Note that if the edge $\{a, a_0\}$ lies in the interior of a fundamental domain, then $A_i$ may only be part of a fundamental domain for some $i$. Suppose the $A_i$ are labeled so that $A_1$ contains $\{a,a'\}$, $A_n$ contains $\{a_0, a_0'\}$, and $A_i$ and $A_{i+1}$ intersect in an edge $\{\epsilon, a_i\}$ of the tiling by fundamental domains where $\epsilon \in \{a,a_0\}$ and $1 \leq i \leq m-1$, as illustrated in Figure \ref{maketile2}. Let $B_1, \ldots, B_n$ and $b_1, \ldots, b_{n-1}$ be similar. Form an embedded cycle \[C_A = \{a, a', p_1, a_1, p_2, a_2, \ldots, a_{m-1}, p_m, a_0', a_0 \},\] where $p_i$ is an embedded path in $A_i$ containing the remaining vertices of $\partial A_i$, but choosing only one vertex from a branching edge of $A_i$. Let \[C_B = \{b, b', q_1, b_1, q_2, b_2, \ldots, b_{n-1}, q_n, b_0', b_0 \}\] be similar. If $|C_A| = |C_B|$, continue to the cell and map definitions. Otherwise, suppose without loss of generality, $|C_B|>|C_A|$. By the choice of fundamental domains, there is a non-branching edge of a fundamental domain in the cycle $C_A$ disjoint from $\{a', a, a_0, a_0'\}$ and its adjacent edges, which can be used to extend the cycle $C_A$ as with the first cell. 
 After extending the cycle if necessary, attach $2$-cells to these boundary cycles to form polygons $P_A$ and $P_B$. 
 
 Map $P_A$ to $P_B$ by a cellular homeomorphism, sending $\{a', a, a_0, a_0'\}$ to $\{b', b, b_0, b_0'\}$, and dilating along each edge of the tile. As before, if $P_A$ or $P_B$ is not convex, subdivide $P_A$ and $P_B$ isomorphically into convex polygons so the configurations have isomorphic $1$-skeleta. The map $P_A \rightarrow P_B$ extends the map $P_V \rightarrow P_W$ and the cellular isomorphism. 
 
 By construction, $P_A,P_B \in \mathcal{P}$. Continue construction in this way along each edge of each polygon constructed. The cell complexes built in the regions $H_1$ and $J_1$ are exhaustive since the tiling of these regions by the fundamental domains $F_1$ and $D_1$, respectively, is exhaustive. That is, in our cell decomposition of $H_1$, the first polygon contains the fundamental domain $F_1$, the next round of polygons contain all of the translates of the fundamental domain $F_1$ that are adjacent to $F_1$, the following round of polygons contain all of the translates of $F_1$ adjacent to these fundamental domains, and so on; the cell decomposition of $J_1$ is similar. 
 
 \begin{center}
  {\it Extending the cell decomposition to the entire universal covers:}
 \end{center}

   First, realize $H_i$ and $J_i$ as isomorphic cell complexes for $2 \leq i \leq 4$ in the same manner as with $H_1$ and $J_1$. Let \[\phi_i: H_i \rightarrow J_i\] be the cellular homeomorphism constructed, which is dilation with the same constant when restricted to each boundary geodesic of $H_i$. So, the maps $\phi_i: H_i \rightarrow J_i$ and $\phi_j : H_j \rightarrow J_j$ agree when restricted to their intersection. We will use the action of the group to extend these maps and hence these cell decompositions to all of $\widetilde{X}_1$ and $\widetilde{X}_2$. 
   
   Recall, $\mathcal{L}_i = \{g\cdot \tg_i \, | \, g \in \pi_1(X_i) \}$ is the set of branching geodesics in $\widetilde{X}_i$. We define a cellular homeomorphism \[\Phi:\widetilde{X}_1 \rightarrow \widetilde{X}_2\] recursively, mapping components of $\mathcal{C}_1 = \widetilde{X}_1 \backslash \mathcal{L}_1$ to components of $\mathcal{C}_2 = \widetilde{X}_2 \backslash \mathcal{L}_2$. 
   
   Let \[\displaystyle \Phi:\bigcup_{i=1}^4 H_i \rightarrow \bigcup_{i=1}^4 J_i\] be defined by the maps above: $\Phi(H_i) = \phi_i(H_i)$. 
   
   Extend the map $\Phi$ along each unmapped branching geodesic of a component mapped during the preceding stage as follows. To begin, let $g\tg$ be a branching geodesic of $H_1$ for some nontrivial $g \in \pi_1(X_1)$. Suppose $R_2$, $R_3$, and $R_4$ are components of $\mathcal{C}_1$ that intersect the boundary of $H_1$ in the branching geodesic $g\tg_1$. Without loss of generality, $g^{-1}(R_i) = H_i$. The isometry $g:H_i \rightarrow R_i$ induces a cell decomposition of $R_i$ isomorphic to the cell decomposition of $H_i$. Suppose $\Phi(g\tg_1) = h\tg_2 \in J_1$ for some $h \in \pi_1(X_2)$. Let $S_2, S_3$, and $S_4$ be the other components of $\mathcal{C}_2$ incident to $h\tg_2$ so that $h^{-1}(S_i) = J_i$. Then, $h:J_i \mapsto S_i$ induces a tiling of $S_i$ isomorphic to the cell decompositions of $J_i$, $H_i$, and $R_i$. Map $R_i$ to $S_i$ by the cellular homeomorphism $h \circ \Phi_i \circ g^{-1}$ for $2 \leq i \leq 4$. 
   
   Repeat this procedure along each unmapped branching geodesic of the regions $H_i$ and $J_i$, then along each unmapped branching geodesic of the regions incident to $H_i$ and $J_i$, and so on to define $\Phi$, an exhaustive cellular homeomorphism $\widetilde{X}_1 \rightarrow \widetilde{X}_2$. By Lemma \ref{tilings}, $\widetilde{X}_1$ and $\widetilde{X}_2$ are bilipschitz equivalent. 
\end{proof}

\begin{corollary} \label{qiclassification}
 {\it If $G, G' \in \mathcal{C}_S$, then $G$ and $G'$ are quasi-isometric. }
\end{corollary}

\section{Analysis of the abstract commensurability classes within $\mathcal{C}_S$}

\subsection{Maximal elements within $\mathcal{C}_S$}

 Let $\mathcal{G} \subset \mathcal{C}_S$ be an abstract commensurability class. A {\it maximal element} for $\mathcal{G}$ is a group $G_0$ that contains every group in $\mathcal{G}$ as a finite-index subgroup. As described below, the existence of a maximal element that lies in $\mathcal{C}_S$ depends on whether the abstract commensurability class contains the fundamental group of a surface identified along a non-separating curve. For this reason, we define the following three subclasses that partition the spaces in $\mathcal{X}_S$ and the groups in $\mathcal{C}_S$. By Theorem \ref{classification}, these subclasses partition the abstract commensurability classes within $\mathcal{C}_S$ as well. 

 \begin{defn} \label{subclasses} 
  \begin{itemize}
   \item Let $\mathcal{X}_0$ be the set of spaces $X \in \mathcal{X}_S$ for which the complement of the singular curve in $X$ consists of four surfaces with one boundary component and unequal genus. Let $\mathcal{C}_0 \subset \mathcal{C}_S$ be the set of fundamental groups of spaces in $\mathcal{X}_0$. 
   
   \item Let $\mathcal{X}_1$ be the set of spaces $X \in \mathcal{X}_S$ for which the complement of the singular curve in $X$ contains either one surface with two boundary components and two surfaces with one boundary component and unequal genus, or, four surfaces, exactly two of which have equal genus. Let $\mathcal{C}_1 \subset \mathcal{C}_S$ be the set of fundamental groups of spaces in $\mathcal{X}_1$. 
   
   \item Let $\mathcal{X}_2$ be the set of spaces $X \in \mathcal{X}_S$ that can be realized as the union of two surfaces along curves of topological type one (see Definition \ref{toptype}). Let $\mathcal{C}_2 \subset \mathcal{C}_S$ be the set of fundamental groups of spaces in $\mathcal{X}_2$. 
  \end{itemize}
 \end{defn}

   {\bf Remark:} In Proposition \ref{maximalelement}, we show that an abstract commensurability class $\mathcal{G} \subset \mathcal{C}_S$ contains a maximal element within $\mathcal{C}_S$ if and only if $\mathcal{G} \subset \mathcal{C}_0$. In Corollary \ref{orbimaximal}, we prove that if $\mathcal{G} \subset \mathcal{C}_2$, then there is a maximal element for $\mathcal{G}$ within the class of right-angled Coxeter groups. For $\mathcal{G} \subset \mathcal{C}_1$, it is not known whether there exists a maximal element for the abstract commensurability class $\mathcal{G}$. 
 
 To construct covers of surfaces glued along separating curves, we use the following lemma, which is a converse to Lemma \ref{eulerchar} for hyperbolic surfaces with one boundary component.

\begin{lemma} \label{boundary} {\it
 For $g_i \geq 1$, if $\chi(S_{g_2,1}) = n\chi(S_{g_1,1})$, then $S_{g_2,1}$ $n$-fold covers $S_{g_1,1}$. }
\end{lemma}
\begin{proof}
 Let 
\begin{equation*}
 \pi_1(S_{g_1,1}) = \left\langle a_1, b_1, \ldots, a_{g_1}, b_{g_1} | \, \, \,  \right\rangle \cong F \,_{2g_1}
\end{equation*}
be a presentation for the fundamental group of $S_{g_1,1}$. The homotopy class of the boundary element $\gamma_1:S^1 \rightarrow S_{g_1,1}$ corresponds to the element $[a_1,b_1] \ldots [a_{g_1},b_{g_1}] \in \pi_1(S_{g_1,1})$.

We exhibit $\pi_1(S_{g_2,1})$ as an index $n$ subgroup of $\pi_1(S_{g_1,1})$ so that in the corresponding cover, $\gamma_1$ has preimage a single curve that $n$-fold covers $\gamma_1$. 

\begin{figure}[t]
\setlength{\abovecaptionskip}{-75pt}
 \begin{tikzpicture}
 \scalebox{.75}{
\node (1) at (canvas polar cs:angle=0,radius=3.5cm)[circle,draw=black!,fill=blue!20] {$0 $};
\node (2) at (canvas polar cs:angle=51.4,radius=3.5cm)[circle,draw=black!,fill=blue!20]{$1 $};
\node (3) at (canvas polar cs:angle=102.85,radius=3.5cm)[circle,draw=black!,fill=blue!20]{$2 $};
\node (4) at (canvas polar cs:angle=154.2,radius=3.5cm)[circle,draw=black!,fill=blue!20]{$3 $};
\node (5) at (canvas polar cs:angle=205.6,radius=3.5cm)[circle,draw=black!,fill=blue!20]{$4$};
\node (6) at (canvas polar cs:angle=257,radius=3.5cm)[circle,draw=black!,fill=blue!20]{$5 $};
\node (7) at (canvas polar cs:angle=308.4,radius=3.5cm)[circle,draw=black!,fill=blue!20]{$6 $};
\node (Y) at (0,-4.8) {};
   \node (A) at (0,-6.3) {};
   \node (Z) at (0, -8.4)[circle,draw=black!,fill=blue!30] {$\,\,\,\,\,$};
 \path (1) edge[->] node {$a_2$} (2);
 \path (2) edge[->] node {$a_2$} (3);
 \path (3) edge[->] node {$a_2$} (4);
 \path (4) edge[->] node {$a_2$} (5);
 \path (5) edge[->] node {$a_2$} (6);
 \path (6) edge[->, bend right = 7] node {$a_2$} (7);
 \path (7) edge[->] node {$a_2$} (1);
 \path (1) edge[->, bend right = 45] node {$b_1$} (2);
 \path (2) edge[->, bend right = 40] node {$b_1$} (3);
 \path (3) edge[->, bend right = 45] node {$b_1$} (4);
 \path (4) edge[->, bend right = 40] node {$b_1$} (5);
 \path (5) edge[->, bend right = 45] node {$b_1$} (6);
 \path (6) edge[->, bend right = 45] node {$b_1$} (7);
 \path (7) edge[->, bend right = 45] node {$b_1$} (1);
 \path (1) edge[->, bend left = 45] node {$b_2$} (2);
 \path (2) edge[->, bend left = 35] node {$b_2$} (3);
 \path (3) edge[->, bend left = 40] node {$b_2$} (4);
 \path (4) edge[->, bend left = 35] node {$b_2$} (5);
 \path (5) edge[->, bend left = 45] node {$b_2$} (6);
 \path (6) edge[->, bend left = 31] node {$b_2$} (7);
 \path (7) edge[->, bend left = 45] node {$b_2$} (1);
 \path (2) edge[->, bend right=75] node[swap] {$a_1$} (3);
 \path (4) edge[->, bend right=75] node [swap] {$a_1$} (5);
 \path (6) edge[->, bend right=75] node [swap]{$a_1$} (7);
\path (2) edge[<-, bend left=75] node {$a_1$} (3);
 \path (4) edge[<-, bend left=75] node {$a_1$} (5);
 \path (6) edge[<-, bend left=75] node {$a_1$} (7);
 \path (1) edge[out=35, in = 325
                , looseness=0.6, loop
                , distance=2cm, ->]
            node {$a_1$}(1);
   \path (Z) edge[out=25, in = 335
                , looseness=0.6, loop
                , distance=2cm, ->]
            node {$a_1$}(Z);
   \path (Z) edge[out=65, in = 115
                , looseness=0.6, loop
                , distance=2cm, ->]
            node [swap] {$a_2$}(Z); 
 \path (Z) edge[out=155, in = 205
                , looseness=0.6, loop
                , distance=2cm, ->]
            node [swap] {$b_1$}(Z);
 \path (Z) edge[out=245, in = 295
                , looseness=0.6, loop
                , distance=2cm, ->]
            node [swap] {$b_2$}(Z);
   \path (Y) edge[->] node {} (A);          
}
\end{tikzpicture}
\caption[A covering map that realizes $S_{11,1}$ as a $7$-fold cover of $S_{2,1}$]{ {\small A covering map that realizes $S_{11,1}$ as a $7$-fold cover of $S_{2,1}$. } }
\label{coveringmap}
\end{figure}
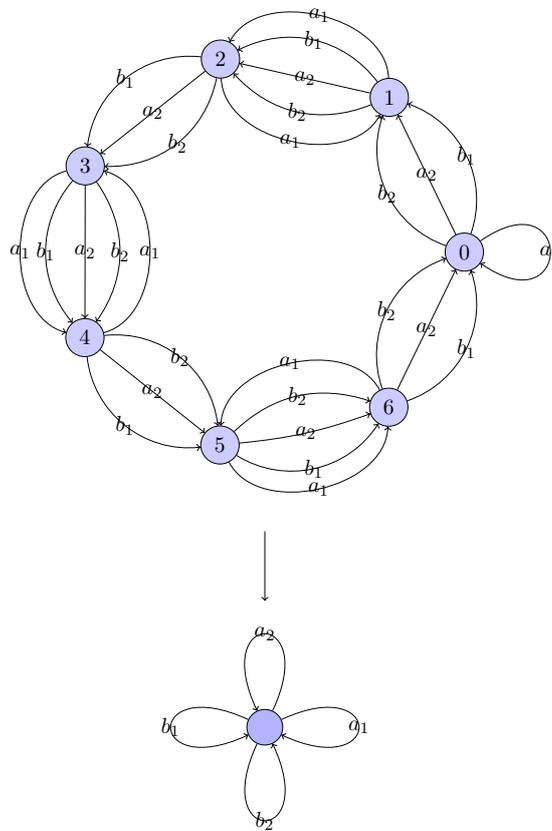

Realize $\pi_1(S_{g_1, 1})$ as the fundamental group of a wedge of $2g_1$ oriented circles labeled by the generating set. Construct an $n$-fold cover of this space as a graph, $\Gamma$, on $n$ vertices labeled $\{0, \ldots, n-1\}$. For every generator besides $a_1$, construct an oriented $n$-cycle on the $n$ vertices with each edge labeled by the generator. Since $\chi(S_{g_1,1})$ and $\chi(S_{g_2,1})$ are both odd, $n$ must be odd as well by Lemma \ref{eulerchar}. Let $\{i, i+1\}$ and $\{i+1,1\}$ be directed edges labeled by $a_1$ for $i<n$ and $i$ odd. Construct a directed loop labeled $a_1$ at vertex $\{0\}$, as illustrated in Figure \ref{coveringmap}. By construction, $\Gamma$ covers the wedge of circles given above. 

To see that $\gamma_1$ has a preimage with one component, choose a vertex $v$ in the graph $\Gamma$ and consider the edge path $p$ with edges labeled $([a_1,b_1]\ldots [a_{g_1},b_{g_1}])^k$, which projects to $\gamma_1$ under the covering map. Then $n$ is the smallest non-zero $k$ for which $p$ terminates at $v$. To see this, note that it suffices to consider the path $p' = [a_1,b_1]^k$ since every other segment $[a_j,b_j]$ returns to its initial vertex. Starting at vertex $\{0\}$, observe that the path $[a_1,b_1]^k$ terminates at the vertex labeled 

 \[
  \begin{cases}
   2k-1 & \text{if } 0<k < \lfloor\frac{n}{2} \rfloor \,\text{mod}\, n\\
   2n-2k    & \text{if } \lfloor \frac{n}{2} \rfloor \leq k <n \,\text{mod} \,n\\
   0  & \text{if} \, k=0\, \text{mod} \,n,
  \end{cases}
\]
proving the claim. 
\end{proof}

{\bf Remark:} Lemma \ref{boundary} may be restated in terms of the {\it Hurwitz realizability problem for branched coverings of surfaces}. In this language, Lemma \ref{boundary} is a special case of \cite[Lemma 7.1]{bogopolskibux}, proved first in \cite{EKS84}, \cite{Hus62}. Lemma \ref{boundary} is included since its proof is new and of independent interest.

In the proof of the characterization of the abstract commensurability classes that contain a maximal element, we will use the following definition. 

\begin{defn}
 If $S_g$ and $S_h$ are closed hyperbolic surfaces, $\gamma$ is a multicurve on $S_g$ and $\rho$ is a multicurve on $S_h$, we say {\it $(S_g, \gamma)$ covers $(S_h,\rho)$} if there exists a covering map $p:S_g \rightarrow S_h$ so that $\gamma$ is the full preimage of $\rho$ in $S_g$. 
\end{defn}
 
 \begin{prop} \label{maximalelement} {\it Let $\mathcal{G} \subset \mathcal{C}_S$ be an abstract commensurability class. 
  \begin{enumerate}[(a)]
   \item There exists a maximal element in $\mathcal{C}_S$ for $\mathcal{G}$ if and only if $\mathcal{G} \subset \mathcal{C}_0$.
   \item If $\mathcal{G} \subset \mathcal{C}_1$, then there exist $G_0, H_0 \in \mathcal{G}$ so that every group in $\mathcal{G}$ is a finite-index subgroup of $G_0$ or $H_0$.
   \item If $\mathcal{G} \subset \mathcal{C}_2$, then there exist $G_0, H_0, K_0, L_0 \in \mathcal{G}$ so that every group in $\mathcal{G}$ is a finite-index subgroup of $G_0$, $H_0$, $K_0$, or $L_0$. 
  \end{enumerate}
   }
 \end{prop}

 \begin{proof}
  We begin by reformulating the statement of the abstract commensurability classification. Let $G \cong \pi_1(X) \in \mathcal{C}_S$ where $X \in \mathcal{X}_S$. Associate a quadruple $(k_1, k_2, k_3, k_4) \in \Z^4$ to $G$ uniquely as follows. 
    \begin{itemize}
      \item If $X$ is the union of four surfaces $S_i$ each with one boundary component, let $k_i = \chi(S_i)$. 
      \item If $X$ is the union of two surfaces $S_1$ and $S_2$ with one boundary component and a surface $S_3$ with two boundary components, let $k_1 = \chi(S_1)$, $k_2 = \chi(S_2)$, and $k_3 = k_4 = \frac{\chi(S_3)}{2}$. 
      \item  If $X$ is the union of two surfaces $S_1$ and $S_2$ each with two boundary components, let $k_1 = k_2 = \frac{\chi(S_1)}{2}$ and $k_3 = k_4 = \frac{\chi(S_2)}{2}$. 
    \end{itemize}
  Relabel the $k_i$ so that $k_i\leq k_j$ if $i\leq j$. By Theorem \ref{classification}, if $G_1\in \mathcal{C}_S$ yields the quadruple $(k_1, \ldots, k_4)$ and $G_2 \in \mathcal{C}_S$ yields the quadruple $(\ell_1, \ldots, \ell_4)$, then $G_1$ and $G_2$ are abstractly commensurable if and only if there exist integers $K$ and $L$ so that $K(k_1, \ldots, k_4) = L(\ell_1, \ldots, \ell_4)$. In other words, each abstract commensurability class in $\mathcal{C}_S$ is characterized by an equivalence class of ordered quadruples, where two quadruples are equivalent if they are equal up to integer scale. 
  
  Suppose first that $\mathcal{G} \subset \mathcal{C}_0$. The maximal element $G_0$ in $\mathcal{G}$ is the group in the abstract commensurability class which yields the quadruple $(p_1, \ldots, p_4)$ where the $p_i$ have no common integer factor. To see that $G_0$ is a maximal element, let $G \cong \pi_1(X) \in \mathcal{G}$ with $X \in \mathcal{X}_S$. Suppose $G$ yields the quadruple $(k_1, \ldots, k_4)$. Then, since the $p_i$ have no common factor, there exists $D \in \N$ so that $D(p_1, \ldots, p_4) = (k_1, \ldots, k_4)$. Since $\mathcal{G} \subset \mathcal{C}_0$, the group $G_0 \cong \pi_1(X_0)$ where $X_0$ consists of four surfaces $S_i$ each with one boundary component and Euler characteristic $p_i$, and, similarly, $G \cong \pi_1(X)$, where $X$ consists of four surfaces $T_i$ each with one boundary component and Euler characteristic $k_i = Dp_i$ for $1\leq i \leq 4$. By Lemma \ref{coveringmap}, $X$ $D$-fold covers $X_0$, so $G$ is a finite-index subgroup of $G_0$ as desired. 
  
  To complete the proof of claim (a), observe that if $\mathcal{G} \not\subset \mathcal{C}_0$, then there are two groups, $H_1$ and $H_2$, in $\mathcal{G}$, where $H_1 \cong \pi_1(S_{h_1}) *_{\la \gamma \ra} \pi_1(S_{h_1'})$ and $H_2 \cong \pi_1(S_{h_2}) *_{\la \rho \ra} \pi_1(S_{h_2'})$ and, up to relabeling, $\gamma \mapsto [\gamma_{h_1}]\in  \pi_1(S_{h_1})$ and $\rho \mapsto [\gamma_{h_2}] \in \pi_1(S_{h_2})$, where $\gamma_{h_1}$ is an essential non-separating simple closed curve and $\gamma_{h_2}$ is a separating simple closed curve. Thus, $(S_{h_1}, \gamma_{h_1})$ and $(S_{h_2}, \gamma_{h_2})$ cannot cover the same pair $(S, \gamma)$, so there is no maximal element in the abstract commensurability class of $G$ in $\mathcal{C}_S$.
  
  Suppose now that $\mathcal{G} \subset \mathcal{C}_1$. Then the groups $G_0$, $H_0 \in \mathcal{G}$ are the two groups in the abstract commensurability class that yield the same quadruple $(p_1, \ldots, p_4)$ where the $p_i$ have no common integer factor. More specifically, since $\mathcal{G} \in \mathcal{C}_1$, $p_i = p_j$ for some $i \neq j$. Let $G_0 \cong \pi_1(X_0)$, where $X_0 \in \mathcal{C}_S$ consists of four surfaces each with one boundary component and Euler characteristic $p_i$. Let $H_0 = \pi_1(Y_0)$, where $Y_0 \in \mathcal{X}_S$ consists of one surface $S$ with Euler characteristic $2p_i$ glued along a non-separating curve $\gamma$ to two surfaces each with one boundary component and Euler characteristic $p_m$ and $p_{\ell}$, respectively, for $m,\ell \neq i, j$. Let $G \in \mathcal{G}$ so $G \cong \pi_1(X)$ and $G$ yields the quadruple $(k_1, \ldots, k_4) = D(p_1, \ldots, p_4)$ for some $D \in \N$. Since $G \in \mathcal{C}_1$, either $X$ consists of four surfaces each with one boundary component, so $G$ is a finite-index subgroup of $G_0$ as before, or, $X$ consists of one surface $T$ with Euler characteristic $2Dp_i$ glued along a non-separating curve $\rho$ to two surfaces each with one boundary component and Euler characteristic $Dp_m$ and $Dp_{\ell}$, respectively. Since there is a (cyclic) $D$-fold cover of $(S, \gamma)$ by $(T, \rho)$, the space $X$ $D$-fold covers $Y_0$ in this case, completing the proof of claim (b). 
  
  Finally, suppose $\mathcal{G} \in \mathcal{C}_2$. In this case, the groups $G_0, H_0, K_0, L_0 \in \mathcal{G}$ are the four groups in the abstract commensurability class that yield the same quadruple $(p_1, \ldots, p_4)$ where the $p_i$ have no common integer factors. Since $\mathcal{G} \subset \mathcal{C}_2$, $p_1 = p_2$ and $p_3 = p_4$. Let $G_0 \cong \pi_1(X_0)$ where $X_0 \in \mathcal{X}_S$ consists of four surfaces $S_i$ each with one boundary component and Euler characteristic $p_i$ for $1 \leq i \leq 4$. Let $H_0 \cong \pi_1(Y_0)$, where $Y_0 \in \mathcal{X}_S$ consists of a surface with Euler characteristic $2p_1$ glued along a non-separating curve to two surfaces each with one boundary component and Euler characteristic $p_3$. Let $K_0 \cong \pi_1(Z_0)$, where $Z_0 \in \mathcal{X}_S$ consists of a surface with Euler characteristic $2p_3$ glued along a non-separating curve to two surfaces each with one boundary component and Euler characteristic $p_1$. Finally, let $L_0 \cong \pi_1(W_0)$, where $W_0 \in \mathcal{X}_S$ consists of a surface with Euler characteristic $2p_1$ and a surface with Euler characteristic $2p_3$ glued to each other along a non-separating curve in each. As above, if $G \cong \pi_1(X) \in \mathcal{G}$ and $X \in \mathcal{X}$, then $X$ finitely covers one of $X_0$, $Y_0$, $Z_0$, or $W_0$, depending on the non-separating curves in $X$, which concludes the proof of (c).  \end{proof}

\subsection{Right-angled Coxeter groups and the Crisp--Paoluzzi examples}

In this section, we discuss the relationship between groups in $\mathcal{C}_S$ and the class of right-angled Coxeter groups. We begin with the relevant background for this section. 


\begin{defn} Let $\Gamma$ be a finite simplicial graph. The {\it right-angled Coxeter group with defining graph $\Gamma$} is 
  \begin{center}
     $W(\Gamma)  = \la v \in V(\Gamma) \, | \, v^2 = 1 \text{ if } v \in V(\Gamma), \, [v,w] = 1 \text{ if } \{v,w\} \in E(\Gamma) \ra$.
  \end{center}
\end{defn}

 For more on right-angled Coxeter groups, see \cite{davis}. As shown in \cite{green}, a right-angled Coxeter group is defined up to isomorphism by its defining graph; that is, $W(\Gamma) \cong W(\Gamma')$ if and only $\Gamma \cong \Gamma'$. Often, group theoretic properties of $W(\Gamma)$ correspond to graph theoretic properties of $\Gamma$. Classic results relevant to our setting are recorded below. 

 \begin{prop}  Let $\Gamma$ be a simplicial graph. 
 \begin{enumerate}
 \item \cite[Pg. 123]{gromov} The group $W(\Gamma)$ is word-hyperbolic if and only if every $4$-cycle in $\Gamma$ has a chord.
 \item \cite[Lemma 8.7.2]{davis} The group $W(\Gamma)$ is one-ended if and only if $\Gamma$ is not a complete graph and there does not exist a complete subgraph $K$ of $\Gamma$ such that $\Gamma \backslash K$ is disconnected. 
 \end{enumerate}
 \end{prop}

  An {\it orbifold} is a topological space $\mathcal{O}$ in which each point has a neighborhood modeled on $\tilde{U}/G$, where $\tilde{U}$ is an open ball in $\R^n$ and $G$ is a finite subgroup of $SO(n)$.  Associated to each point in the orbifold is the finite group $G$ called its {\it isotropy group}. A point is called a {\it ramification point} if its isotropy group is non-trivial. The set of all ramification points is called the {\it ramification locus} of the orbifold. The underlying topological space of an orbifold $\mathcal{O}$ is denoted $|\cO|$. Background and a more formal definition of orbifolds can be found in \cite[Chapter 6]{kapovich} and \cite[Chapter 13]{ratcliffe};  recent applications for commensurability can be found in the survey paper \cite{walsh}.
 
 A {\it homeomorphism} between orbifolds $\cO$ and $\mathcal{R}$ is a homeomorphism $h: |\cO| \rightarrow |\mathcal{R}|$ such that for each point $x \in \cO$, $y = h(x) \in \mathcal{R}$, there are coordinate neighborhoods $U_x = \tilde{U}_x/G_x$ and $V_y = \tilde{V}_y/G_y$ such that $h$ lifts to an equivariant homeomorphism $\tilde{h}_{xy}:\tilde{U}_x \rightarrow \tilde{V}_y$. An {\it orbi-complex} is a disjoint union of orbifolds identified to each other along homeomorphic suborbifolds. 
 
 An {\it orbifold covering} $p: \mathcal{O}' \rightarrow \mathcal{O}$ is a continuous map $|\mathcal{O}'| \rightarrow |\mathcal{O}|$ such that if $x \in \mathcal{O}$ is a ramification point with neighborhood given by $U = \tilde{U}/G$, then each component $V_i$ of $f^{-1}(U)$ is isomorphic to $\tilde{U}/G_i$ where $G_i \leq G$ and $p|_{V_i}:V_i \rightarrow U$ is $\tilde{U}/G_i \rightarrow \tilde{U}/G$. The {\it universal covering} $p:\tilde{\mathcal{O}} \rightarrow \mathcal{O}$ is a covering such that for any other covering $p':\mathcal{O}' \rightarrow \mathcal{O}$ there exists a covering $\tilde{p}:\tilde{\mathcal{O}} \rightarrow \mathcal{O}'$ such that $p' \circ \tilde{p} = p$. The group of {\it deck transformations} of the orbifold covering $p: \mathcal{O}' \rightarrow \mathcal{O}$ is the group of self-diffeomorphisms $h:\mathcal{O}' \rightarrow \mathcal{O}'$ such that $p \circ h = p$. The {\it orbifold fundamental group}, $\pi_1^{orb}(\mathcal{O})$, is the group of deck transformations of its universal covering. Then $\mathcal{O} = \tilde{\mathcal{O}} / \pi_1^{orb}(\mathcal{O})$. The orbifold $\cO$ is called a {\it reflection orbifold} if $\pi_1(\cO)$ is generated by reflections. The orbifold fundamental group can also be defined based on homotopy classes of loops in $\cO$; this definition appears in \cite[Chapter 13]{ratcliffe}. A form of the Seifert-Van Kampen theorem allows one to compute the fundamental group of orbifolds; see Section 2 of \cite{scott}.

  \begin{figure}[t]
   \includegraphics[height=6.0cm]{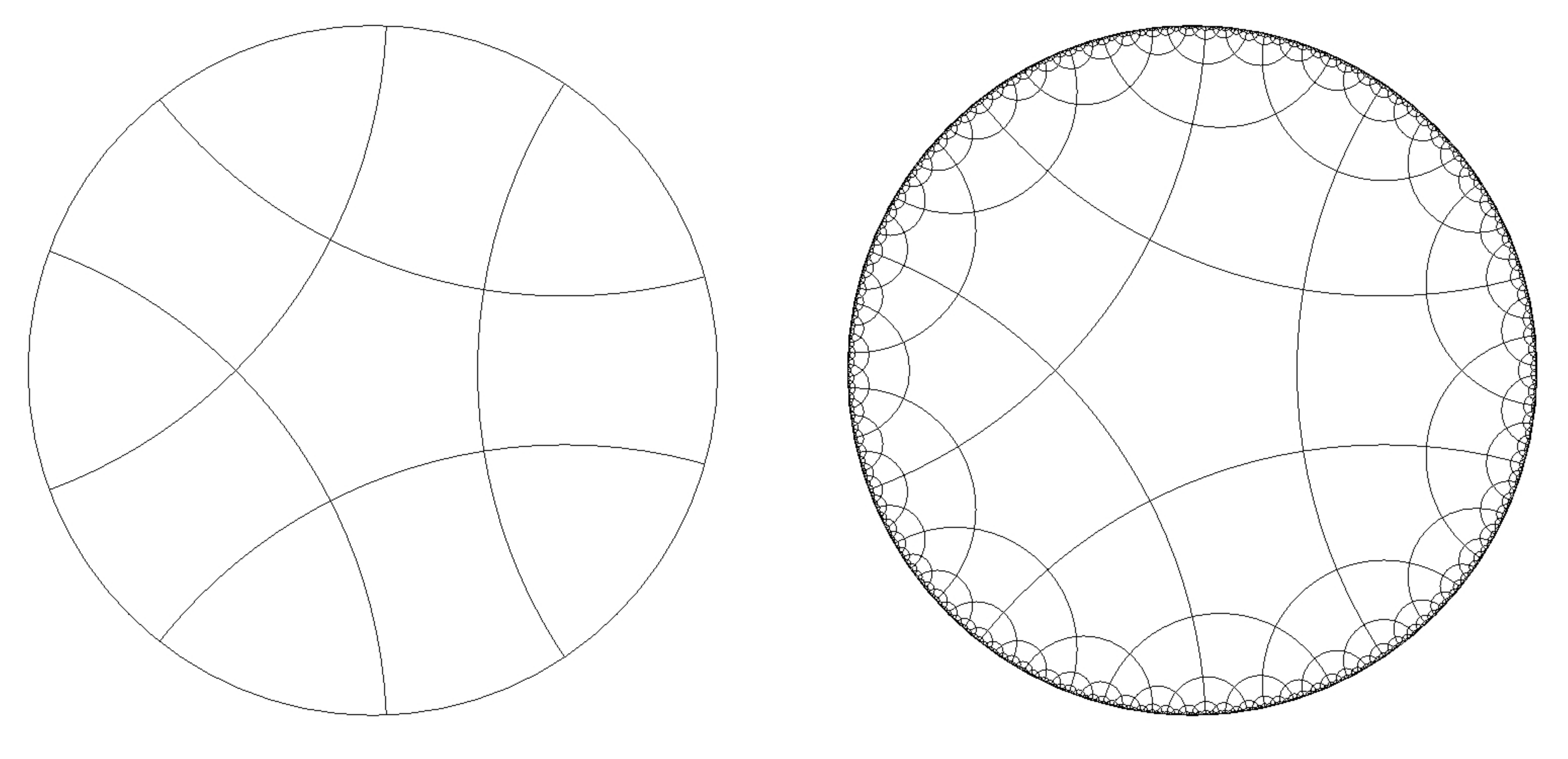}
   \caption[The orbit of a pentagon under the action of a right-angled Coxeter group on the hyperbolic plane]{ {\small On the left are five geodesic lines in the disk model of the hyperbolic plane; on the right, is their orbit under the action of the right-angled Coxeter group $W_5$. Both figures were drawn with Curt McMullen's lim program \cite{mcmullen}.  } }
  \label{pentagons}
 \end{figure}
  
 Hyperbolic surfaces finitely cover reflection orbifolds, so hyperbolic surface groups are finite-index subgroups of right-angled Coxeter groups. More specifically, let $W_n$ be the right-angled Coxeter group with defining graph an $n$-cycle. If $n \geq 5$, $W_n$ acts geometrically on the hyperbolic plane: $W_n$ is isomorphic to the group generated by reflections about the geodesic lines through the $n$-sides of a right-angled hyperbolic $n$-gon. One such example is given in Figure \ref{pentagons}. Let $\mathcal{O}_n$ denote the quotient of the hyperbolic plane under the action of $W_n$ so $\pi_1^{orb}(\mathcal{O}_n) \cong W_n$. Every closed orientable surface of genus greater than one finitely covers $\mathcal{O}_5$ (for example, see \cite{scott78}), so $\pi_1(S_g)$ is a finite-index subgroup of $W_5$ for $g \geq 2$. 
 
 As orbifolds, $\mathcal{O}_n$ and $\mathcal{O}_m$ may be identified to each other along homeomorphic suborbifolds to form an orbi-complex. If the suborbifolds each have underlying space a geodesic segment that meets the boundary edges of the reflection orbifolds at right angles, then the orbi-complex obtained has orbifold fundamental group a right-angled Coxeter group. There are two homeomorphism types of such suborbifolds of $\mathcal{O}_n$: a reflection edge and the geodesic segment that connects the interior of reflection edges that are separated from each other by at least two reflection edges on either side. 
 
 The orbi-complex obtained by identifying $\mathcal{O}_n$ and $\mathcal{O}_m$ along a reflection edge in each is denoted $\mathcal{O}_{m,n}$. The orbifold fundamental group of $\mathcal{O}_{m,n}$ is the right-angled Coxeter group $W_{m,n}$ introduced by Crisp--Paoluzzi in \cite{crisppaoluzzi}, and is defined as follows. 
 
 \begin{defn} \cite{crisppaoluzzi} For $m,n \geq 5$, define $W_{m,n} = W(\Gamma_{m,n})$, where $\Gamma_{m,n}$ denotes the graph which consists of a circuit of length $m$ and a circuit of length $n$ identified along a common subpath of edge-length $2$.  
 \end{defn}

 {\bf Remark:}  Our notation for $W_{m,n}$ varies slightly from that given in \cite{crisppaoluzzi}; they define $\Gamma_{m,n}$ as the graph which consists of a circuit of length $m+4$ and a circuit of length $n+4$ identified along a common subpath of edge-length $2$ and $m,n \geq 1$. One can easily translate between the two notations. 
 
 On the other hand, the orbi-complex obtained by identifying $\mathcal{O}_n$ and $\mathcal{O}_m$ along geodesics connecting reflection edges separated from each other by at least two reflection edges on either side can also be viewed as the union of four right-angled reflection orbifolds with one boundary edge identified to each other along their boundary edges. The orbifold fundamental group of each component orbifold with boundary is $P_n$, the right-angled Coxeter group with underlying graph a path of length $n$ for some $n \geq 4$. More specifically, for $n \geq 4$, $P_n$ acts properly discontinuously by isometries on the hyperbolic plane by reflecting about $n$ geodesic lines, whose intersection graph is a path of length $n$ and so that the intersecting lines meet at right angles; an example is illustrated in Figure \ref{p4}. The quotient of the hyperbolic plane under the group $P_n$ is an open infinite-area right-angled hyperbolic reflection orbifold. Truncate this space along the unique geodesic in the homotopy class of the boundary to obtain the orbifold $\mathcal{O}_{n,1}$, a compact orbifold with boundary and $\pi_1^{orb}(\mathcal{O}_{n,1}) = P_n$. 
   
    \begin{figure}[t]
   \includegraphics[height=6.0cm]{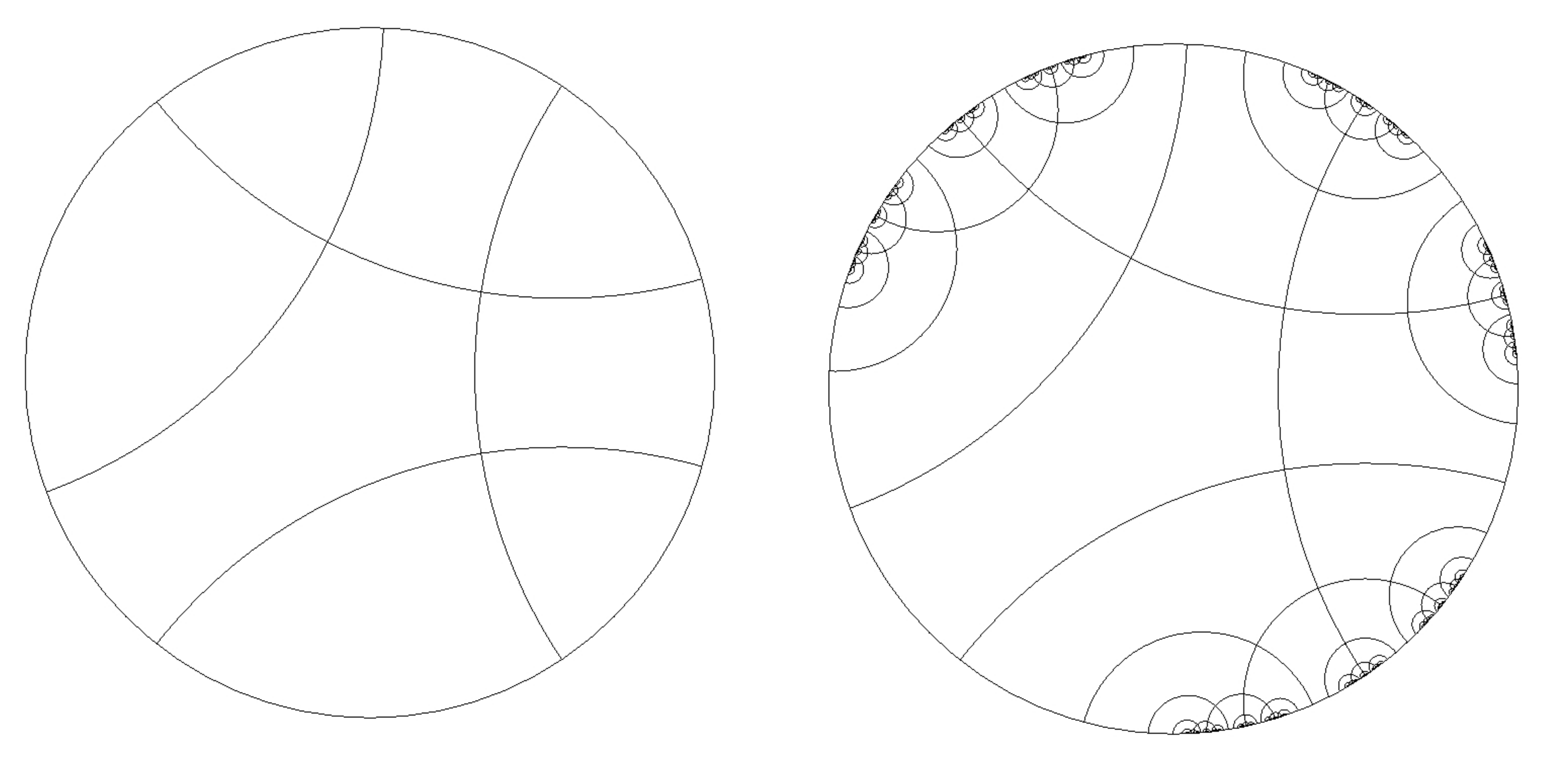}
   \caption[The orbit of four lines under the action of a right-angled Coxeter group on the hyperbolic plane]{ {\small On the left are four geodesic lines in the disk model of the hyperbolic plane; on the right, is their orbit under the action of the right-angled Coxeter group with underlying graph a path of length four. Both figures were drawn with Curt McMullen's lim program \cite{mcmullen}.  } }
  \label{p4}
 \end{figure} 
 
  For $n_i \geq 4$, the orbifolds $\mathcal{O}_{n_1,1}, \ldots, \mathcal{O}_{n_4,1}$ may be identified along their boundary curves to form an orbi-complex we denote $\mathcal{O}(n_1, \ldots, n_4)$. The orbifold fundamental group of the orbi-complex $\mathcal{O}(n_1, \ldots, n_4)$ is the right-angled Coxeter group with underlying graph denoted $\Theta(n_1, \ldots, n_4)$ that consists of four paths of length $n_i \geq 4$ glued to each other along their endpoints. The graphs $W_{m,n}$ and $\Theta(n_1, \ldots, n_4)$ are examples of {\it generalized $\Theta$-graphs}, which were introduced by Dani--Thomas in \cite{danithomas}, and which are defined more formally below. 
 
 \begin{defn} \label{gentheta}
  Let $k \geq 3$, $n_1 \geq 3$ and $n_2, \ldots, n_k \geq 4$ be integers. Let $\Psi_k$ be the graph with two vertices $a$ and $b$ and $k$ edges $e_1, \ldots, e_k$ connecting the vertices $a$ and $b$. The {\it generalized $\Theta$-graph} $\Theta(n_1, \ldots, n_k)$ is obtained by subdividing the edge $e_i$ of $\Psi_k$ into $n_i-1$ edges by inserting $n_i-2$ new vertices along $e_i$ for $1 \leq i \leq n$. 
 \end{defn}
 
 {\bf Remark:} Each right-angled Coxeter group with defining graph a generalized $\Theta$-graph is the orbifold fundamental group of a right-angled hyperbolic reflection orbi-complex of one of two types that generalize the orbi-complexes described above. That is, if $n_1 = 3$, the associated orbi-complex is similar to $\mathcal{O}_{m,n}$: it consists of $k-1$ right-angled hyperbolic reflection orbifolds identified to each other along a reflection edge in each. If $n_1 >3$, the associated orbi-complex is similar to $\mathcal{O}(n_1, \ldots, n_4)$: it consists of $k$ right-angled hyperbolic reflection orbifolds with boundary identified to each other along their boundary edges. In upcoming joint work with Pallavi Dani and Anne Thomas, we characterize the abstract commensurability classes in these settings. 
 
 {\bf Remark:} In this section, we prove that the fundamental group of two surfaces identified along separating curves is a finite-index subgroup of a right-angled Coxeter group with defining graph $\Theta(n_1, \ldots, n_4)$ for $n_i \geq 4$. We prove the fundamental group of two surfaces identified along curves of topological type one (see Definition \ref{toptype}) is a finite-index subgroup of the right-angled Coxeter group $W_{m,n}$ with defining graph $\Theta(3, n_1, n_2)$ and $n_i \geq 4$. It remains open whether the fundamental group of the union of two surfaces obtained by gluing a non-separating curve to a curve that separates the surface into two subsurfaces of unequal genus is a finite-index subgroup of a right-angled Coxeter group. 
  
 Using the following lemma, we prove that in every abstract commensurability class of a group in $\mathcal{C}_S$ there is a group that is a finite-index subgroup of a right-angled Coxeter group with underlying graph $\Theta(n_1, \ldots, n_4)$ and $n_i \geq 4$. 

 \begin{lemma} \label{firacg} {\it If $S_1, \ldots, S_k$ are orientable hyperbolic surfaces with one boundary component, identified to each other along their boundary components to form the space $X$, then $\pi_1(X)$ is a finite-index subgroup of a right-angled Coxeter group.   
   }
 \end{lemma}

 \begin{proof}
  We prove $X$ four-fold covers the reflection orbi-complex $\mathcal{O}(n_1, \ldots, n_k)$ for some $n_i \geq 4$ whose orbifold fundamental group is a right-angled Coxeter group with underlying graph the generalized $\Theta$-graph $\Theta(n_1, \ldots, n_k)$.

  \begin{figure}[t]
   \includegraphics[height=6.7cm]{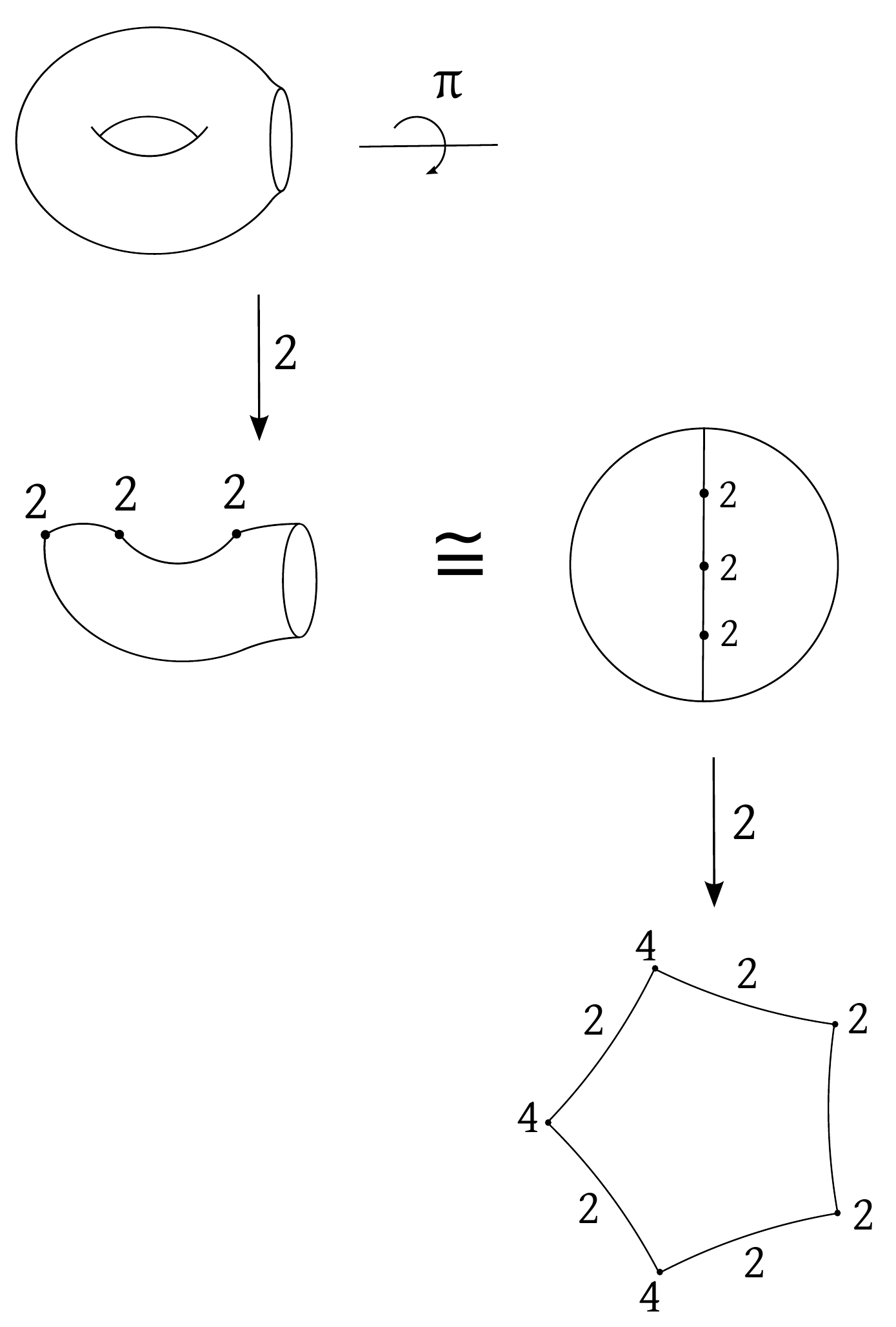}
   \caption[A $4$-fold cover of an orbifold by the surface $S_{1,1}$]{ {\small Illustrated above is a $4$-fold cover of the orbifold $\mathcal{O}_{4,1}$ by the surface with boundary $S_{1,1}$.} }
  \label{pnorbi}
 \end{figure}

  The surface with boundary $S_i \subset X$ four-fold covers $\mathcal{O}_{n_i, 1}$ for some $n_i \geq 4$ such that the boundary of $S_i$ four-fold covers the boundary edge of $\mathcal{O}_{n_i,1}$ as illustrated in Figure \ref{pnorbi}. To see this, skewer $S_i$ through its boundary component so that $2g_i+1$ points on the surface intersect the skewer, and rotate by $\pi$. The quotient is homeomorphic to a disk with $2g_i +1$ cone points of order two, which may be arranged on the diameter of the disk. Reflection across the diameter gives the desired covering map $S_i \rightarrow \mathcal{O}_{n_i,1}$. Thus, the union of these surfaces $S_i$ glued along their boundary curves four-fold covers the union of the orbifolds along their boundary lines concluding the proof.  
  \end{proof}

 \begin{corollary} \label{ACracg}
   {\it If $G \in \mathcal{C}_S$, then $G$ is abstractly commensurable to a right-angled Coxeter group. }
 \end{corollary}
 \begin{proof}
  Let $G \in \mathcal{C}_S$. By the abstract commensurability classification within $\mathcal{C}_S$ given in Theorem \ref{classification}, there exists $Y\in \mathcal{X}_S$ whose fundamental group is abstractly commensurable to $G$ and so that $Y$ has one singular curve that identifies the boundary components of four surfaces each with one boundary component. The group $\pi_1(Y)$ is a finite-index subgroup of a right-angled Coxeter group by Lemma \ref{firacg}, so, $G$ is abstractly commensurable to a right-angled Coxeter group.  
 \end{proof}

 For the remainder of the section, we restrict attention to the relationship between the groups in $\mathcal{C}_S$ and the groups $W_{m,n}$ studied by Crisp--Paoluzzi in \cite{crisppaoluzzi}. Recall, $\mathcal{X}_2 \subset \mathcal{X}_S$ is defined to be the set of spaces $X \in \mathcal{X}_S$ that can be realized as the union of two surfaces along curves of topological type one. The groups $\mathcal{C}_2 \subset \mathcal{C}_S$ are the fundamental groups of spaces in $\mathcal{X}_2$ (see Definition \ref{subclasses}).


\begin{lemma} \label{orbicovers}
 {\it If $X = S_g \cup_{\gamma} S_h \in \mathcal{X}_2$, then $X$ $8$-fold covers $\mathcal{O}_{g+3,h+3}$. Conversely, if $m,n \geq 5$, then $\mathcal{O}_{m,n}$ is $8$-fold covered by $S_{m-3} \cup_{\gamma}S_{n-3} \in \mathcal{X}_2$.  }
\end{lemma}

\begin{proof}
 We show that if $\gamma_g:S^1 \rightarrow S_g$ is an essential simple closed curve of topological type one, then there exists an $8$-fold orbifold covering map $S_g \rightarrow \mathcal{O}_{g+3}$ so that $\gamma_g$ orbifold covers a reflection edge by degree $8$, as illustrated in Figure \ref{T1CPcovers}. Thus, if $X = S_g \cup_{\gamma} S_h$, where $\gamma$ identifies two curves of topological type one, then $S_g \cup_{\gamma} S_h$ $8$-fold orbifold covers $\mathcal{O}_{g+3, h+3}$. 
 
 First suppose $\gamma_g:S^1 \rightarrow S_g$ is non-separating. Skewer $S_g$ so that $2g+2$ points on the surface intersect the skewer, and rotate by $\pi$. The quotient under this action is $S^2(2, \ldots, 2)$, the $2$-sphere with $2g+2$ cone points of order two. This map $p_1:S_g \rightarrow S^2(2, \ldots, 2)$ is an orbifold covering map: each ramification point in the sphere has a neighborhood in which the cover is given by rotation by $\pi$, and all other points have a neighborhood with preimage two homeomorphic copies of the neighborhood. The six cone points may be arranged along the equator of the sphere. Reflection through the equatorial plane has a quotient $\mathcal{O}_6$. Finally, $\mathcal{O}_6$ $2$-fold orbifold covers $\mathcal{O}_5$ by reflection, which can be seen by unfolding $\mathcal{O}_5$ along a reflection edge. It is clear that this covering, illustrated in Figure \ref{T1CPcovers} can be arranged so that $\gamma_g$ $8$-fold covers a reflection edge. 
 
 \begin{figure}[t]
   \includegraphics[height=10.0cm]{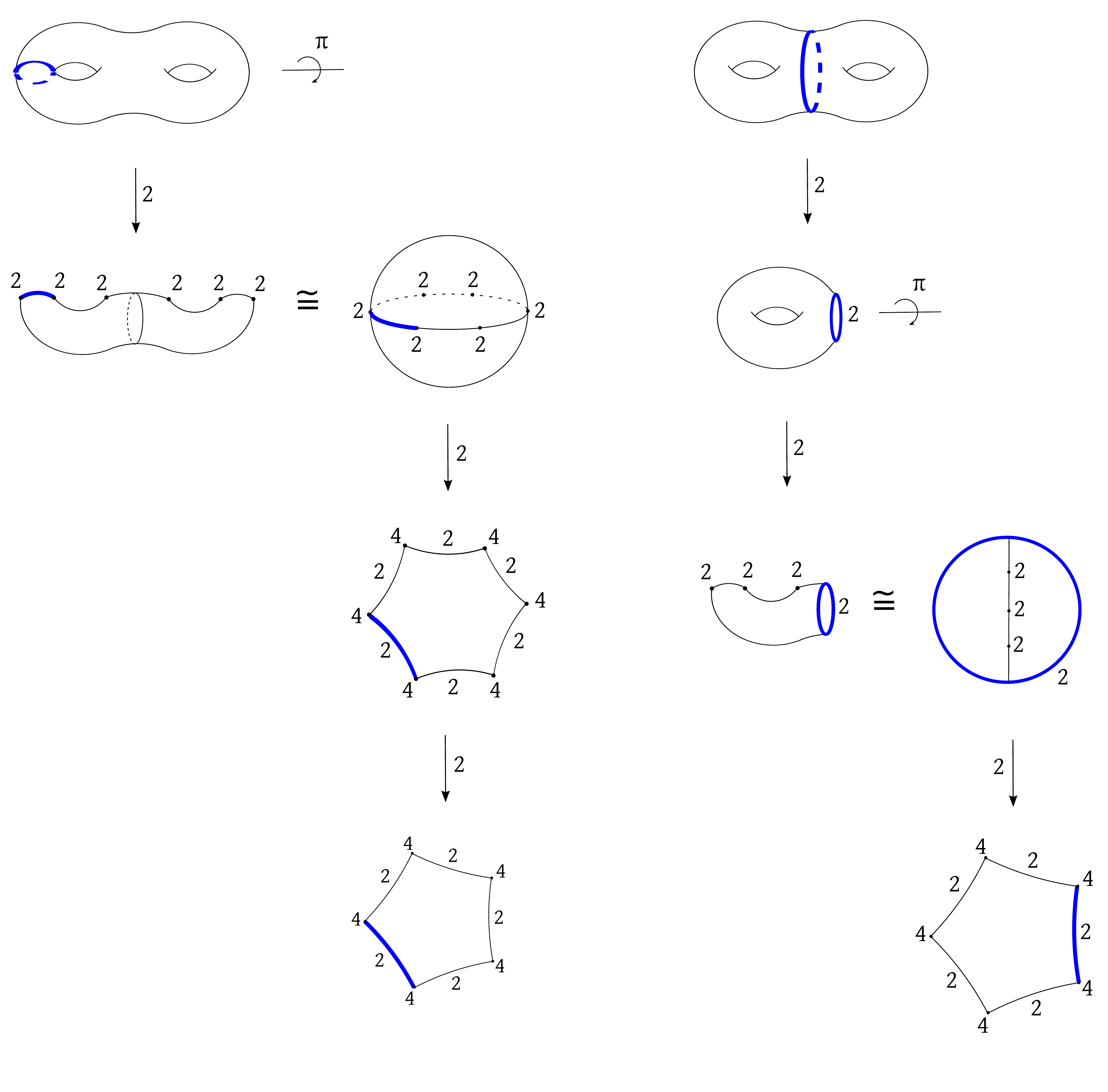}
   \caption[A cover of reflection orbifolds by surfaces]{ {\small Shown above are orbifold covering maps $S_2 \rightarrow \mathcal{O}_5$ described in Lemma \ref{orbicovers} and constructed so that the highlighted curves of topological type one cover a reflection edge in the orbifold $\mathcal{O}_5$. In particular, the union of these surfaces over the highlighted curves finitely covers the union of the orbifolds along the reflection edges.  } }
  \label{T1CPcovers}
 \end{figure}
 
 Now suppose $\gamma_g:S^1 \rightarrow S_g$ is separating. Reflecting $S_g$ across the curve $\gamma_g$ yields a $2$-fold orbifold cover of an orbifold with orbifold boundary and underlying space $S_{\frac{g}{2}, 1}$. Skewer this orbifold along $g+1$ points and rotate by $\pi$ yielding an orbifold with underlying space a disk, $g+1$ cone points or order two, and so that the boundary consists solely of reflection points. Finally arrange the cone points along a diameter of the disk and reflect about this line. These covering maps are illustrated in Figure \ref{T1CPcovers}. As in the non-separating case, one can easily verify each of these maps is an orbifold covering map. \end{proof}
 
 We immediately obtain the following corollary.
 
 \begin{corollary}
  {\it If $G \in \mathcal{C}_2$, then $G$ embeds as a finite-index subgroup in the right-angled Coxeter group $W_{m,n}$ for some $m$ and $n$.}
 \end{corollary}

 {\bf Remark:} An alternative covering map $S_2 \rightarrow \mathcal{O}_5$ appears in \cite{scott78}. Under this covering map, illustrated in Figure \ref{reflectioncovers}, the curves of topological type one can also be chosen to cover a reflection edge in the pentagon orbifold. 
 
\begin{figure}[t]
   \includegraphics[height=10.0cm]{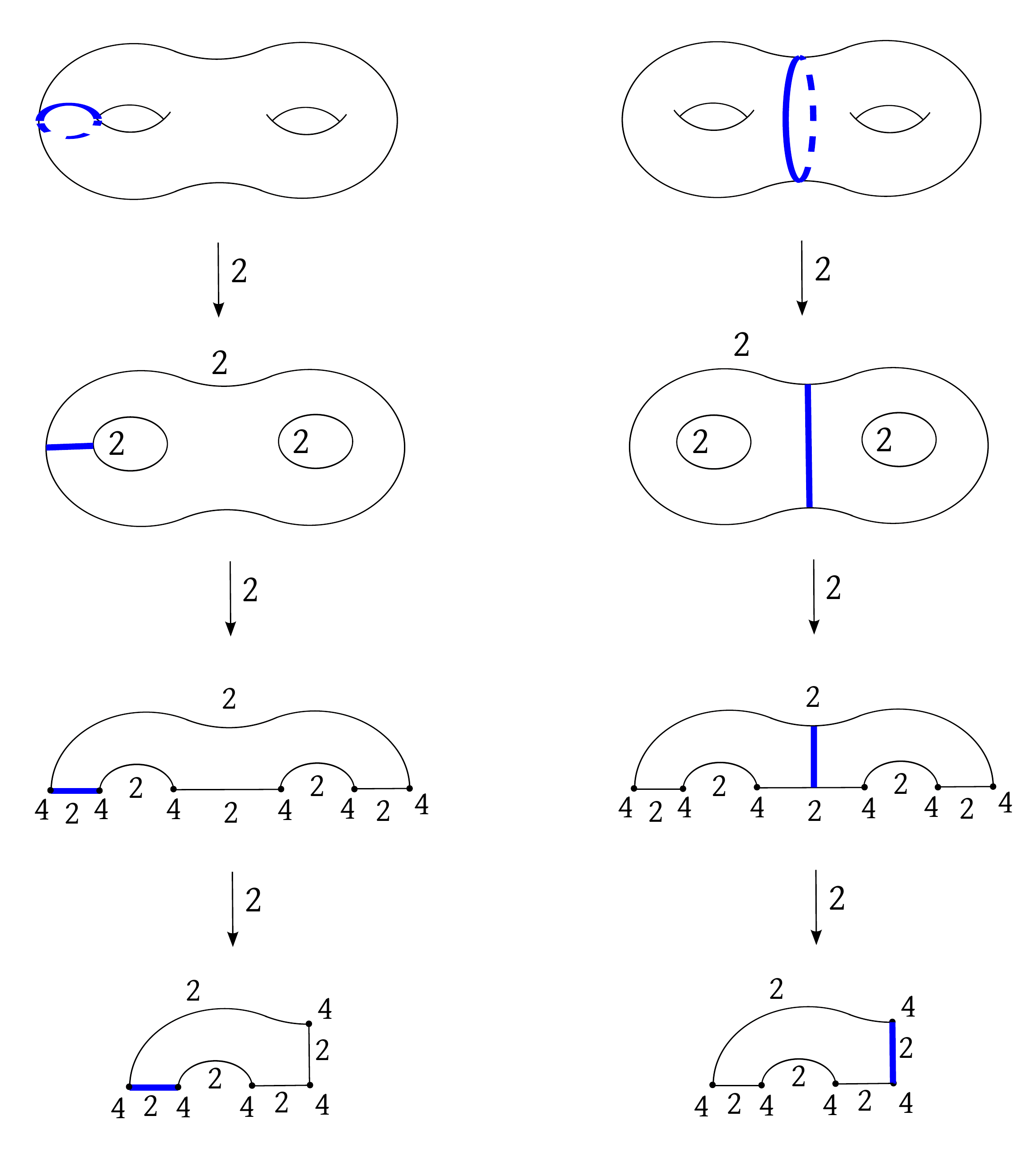}
   \caption[An alternative covering map of reflection orbifolds by surfaces]{ {\small Pictured above are orbifold covering maps that appear in \cite{scott78}. Each map can be realized by embedding the surface in $\R^3$ and reflecting about a plane cutting through the surface. For our purposes, it is important to note that both curves of topological type one cover a reflection edge by degree eight.} }
  \label{reflectioncovers}
\end{figure}

\begin{prop} \label{CPAC}
 {\it If $G \in \mathcal{C}_S$, then $G$ is abstractly commensurable to $W_{m,n}$ for some $m$ and $n$ if and only if $G \in \mathcal{C}_2$. }
\end{prop}
\begin{proof}
 Suppose $G \in \mathcal{C}_2$ so $G \cong \pi_1(X)$ with $X \in \mathcal{X}_2$. By Lemma \ref{orbicovers}, $X$ finitely covers $\mathcal{O}_{m,n}$ for some $m,n$. Hence $G$ is abstractly commensurable to $W_{m,n}$ for some $m$ and $n$. Conversely, suppose $G \in \mathcal{C}_S$ and $G$ is abstractly commensurable to $W_{m,n}$ for some $m$ and $n$. By Lemma \ref{orbicovers}, $W_{m,n}$ is abstractly commensurable to $G'$ for some $G' \in \mathcal{C}_2$. Since abstract commensurability is an equivalence relation, $G$ is abstractly commensurable to $G'$ so $G \in \mathcal{C}_2$ by Theorem \ref{classification}.\end{proof}

Finally, we may use the analysis of this section to produce a maximal element in the class of right-angled Coxeter groups for abstract commensurability classes within $\mathcal{C}_2$.  
 
\begin{corollary} \label{orbimaximal}
 {\it If $G \in \mathcal{C}_2$, then there is a right-angled Coxeter group $G_0$ so that every group in $\mathcal{C}_S$ in the abstract commensurability class of $G$ is a finite-index subgroup of $G_0$. }
\end{corollary}

\begin{proof}
  Let $G \in \mathcal{C}_2$ and let $\mathcal{G} \subset \mathcal{C}_S$ denote the abstract commensurability class of $G$ in $\mathcal{C}_S$. By Lemma \ref{orbicovers},  $G$ is a finite-index subgroup of $W_{m,n}$ for some $m$ and $n$, and, if $G' \in \mathcal{G}$, then $G'$ is a finite-index subgroup of $W_{k, \ell}$ for some $k$ and $\ell$. By \cite[Theorem 1.1]{crisppaoluzzi}, $W_{m,n}$ and $W_{k,\ell}$ are abstractly commensurable if and only if $\frac{m-4}{n-4} = \frac{k-4}{\ell-4}$. Furthermore, $\mathcal{O}_{m,n}$ finitely covers $\mathcal{O}_{p,q}$ whenever $\frac{p-4}{q-4} = \frac{m-4}{n-4}$ and ${\mathsf gcd}(p-4, q-4) = 1$. Thus, $G'$ is a finite-index subgroup of $W_{p,q}$, and $W_{p,q}$ is a maximal element for $\mathcal{G}$ within the class of right-angled Coxeter groups.  \end{proof}

\subsection{Common CAT$(0)$ cubical geometry}
 
  A {\it CAT$(0)$ cube complex} is a polyhedral complex of non-positive curvature whose cells are Euclidean cubes. {\it Special cube complexes}, introduced and defined by Haglund--Wise, are cube complexes in which the hyperplanes are embedded, $2$-sided, and satisfy certain intersection and osculation conditions; a cube complex is special if and only if its fundamental group embeds in a right-angled Artin group \cite{haglundwise}.  For background and details on groups acting on cube complexes, see \cite{sageev}; in particular, details of cubulations of surface groups are given in \cite[Chapter 4.1]{sageev}. 
 
 

  \begin{prop} \label{commoncubing} {\it Let $\mathcal{G} \subset \mathcal{C}_S$ be an abstract commensurability class within $\mathcal{C}_S$. There exists a $2$-dimensional CAT$(0)$ cube complex $X$ so that if $G \in \mathcal{G}$, $G$ acts properly discontinuously and cocompactly by isometries on $X$. Moreover, the quotient $X/G$ is a non-positively curved special cube complex.  }
  \end{prop}

 \begin{proof}
     Let $\mathcal{G} \subset \mathcal{C}_S$ be an abstract commensurability class within $\mathcal{C}_S$. As given in Proposition \ref{maximalelement}, there exists a set of groups $\mathcal{H}(\mathcal{G}) \subset \mathcal{C}_S$ so that every group in $\mathcal{G}$ is a finite-index subgroup of a group in $\mathcal{H}(\mathcal{G})$. So, it suffices to prove that all groups in $\mathcal{H}(\mathcal{G})$ act properly discontinuously and cocompactly by isometries on the same CAT$(0)$ cube complex, with each quotient a special non-positively curved cube complex. 
     
     The set $\mathcal{H}(\mathcal{G}) = \{H_i\} \subset \mathcal{C}_S$ has cardinality one, two, or four, depending on whether $\mathcal{G}$ is in $\mathcal{C}_0$, $\mathcal{C}_1$, or $\mathcal{C}_2$, respectively (see Definition \ref{subclasses}). The groups in $\mathcal{H}(\mathcal{G})$ can be expressed as $H_i \cong \pi_1(X_i)$ with $X_i \in \mathcal{X}_S$ and have the following structure by Proposition \ref{maximalelement}. Each space $X_i$ is the union of closed orientable surfaces $S$ and $T$ along the essential simple closed curves $\gamma_i$ on $S$ and $\rho_i$ on $T$. If $\mathcal{G} \subset \mathcal{C}_0$, the surfaces $S$ and $T$ are identified along separating simple closed curves.  If $\mathcal{G} \subset \mathcal{C}_1$, without loss of generality, in $X_1$, $S$ is glued along a non-separating curve to a separating curve on $T$. In $X_2$, $S$ is glued along a separating curve that divides the surface exactly in half to a separating curve on $T$. Similarly, if $\mathcal{G} \subset \mathcal{C}_2$, $X_1$, $X_2$, $X_3$, and $X_4$ are obtained by gluing surfaces $S$ and $T$ of even genus together, where the four spaces realize the four combinations of gluing $S$ and $T$ along a non-separating or a separating curve that divides the surface exactly in half. By Theorem \ref{classification} (and as illustrated in Figure \ref{coverssamechi}), there exists a space $Y$ that consists of four surfaces each with two boundary components glued to each other along their boundary components so that $Y$ has two singular curves and so that $Y$ $2$-fold covers $X_i$ for all $i$. 
    
     We will first give each surface $S$ and $T$ a special cube complex structure coming from a filling collection of finitely many curves that includes the amalgamated curve; the chosen curves correspond to the set of hyperplanes in the cube complex. We will take the barycentric subdivision of each cube complex, and we will glue the cube complexes together along the locally geodesic paths coming from the amalgamating curves. We show the resulting cube complex obtained after gluing is also special and the cube complex structures on $X_i$ and $X_j$ have the same full pre-image in the $2$-fold cover $Y$ for all $i,j$. Then, we will conclude $\pi_1(X_i)$ and $\pi_1(X_j)$ act properly discontinuously and cocompactly by isometries on the same CAT$(0)$ cube complex.     
    
     To specify the cube complexes, we will first specify a finite filling collection of simple closed curves on $S$ and $T$ satisfying the following: 
     
     \begin{enumerate}
      \item The collection of curves on $S$ includes $\gamma_i$ and the collection of curves on $T$ includes $\rho_i$. Moreover, the collection of curves on $S$ intersects $\gamma_i$ in four points; likewise, the collection of curves on $T$ intersects $\rho_i$ in four points. 
      
      \item For a surface of even genus, the filling collections of curves specified for a non-separating amalgamated curve and for a separating amalgamated curve that divides the surface exactly in half have the same full preimage in the two-fold cover of the surface in the space $Y$. 
      
      \item The cube complex dual to the filling set of curves in $X_i$ is a  $2$-dimensional non-positively curved special cube complex.      
     \end{enumerate}

   \begin{figure}[t] 
   \includegraphics[height=3.5in]{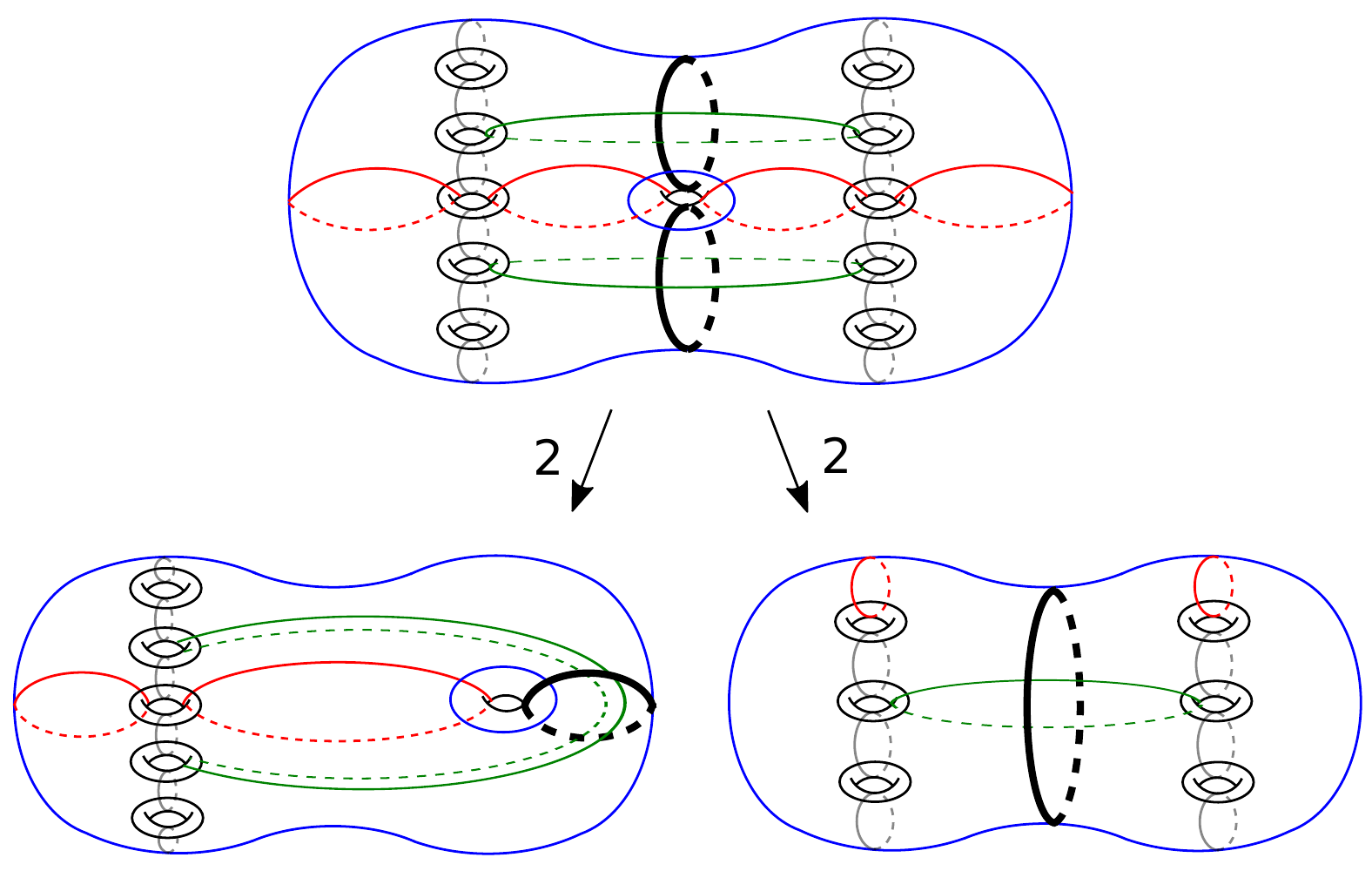}
   \caption{ {\small Illustrated above are the filling collections of simple closed curves used to cubulate the amalgams. On the left is the collection chosen when the gluing curve is non-separating; on the right is the collection chosen when the gluing curve separates the surface exactly in half. The collections yield $2$-dimensional special cube complexes on each surface and have the same full preimage in the two-fold cover illustrated above. } }
  \label{fillings}
 \end{figure}
  
  The filling selection of curves described below is illustrated in an example in Figure \ref{fillings}. To choose the filling collection if $\gamma_i$ or $\rho_i$ is a non-separating curve, arrange the $g$ holes on the surface with $g-1$ holes in one column and one hole in the second. Then, the filling collection of curves contains the following simple closed curves:
  
  \begin{itemize}
   \item The non-separating curve $\gamma_i$ or $\rho_i$, drawn in thick black
   \item A curve around each genus and around the perimeter of the surface, as drawn in blue and black
   \item If the genus is even, include two curves, as drawn in red that, along with the thick non-separating curve, separate the surface exactly in half. If the genus is odd, include one curve that along with the amalgamating curve separates the surface exactly in half
   \item $g$ non-separating curves that connect the $g-1$ holes in the first column, drawn in grey
   \item A curve that intersects the amalgamated curve in two points and passes through two holes on the surface, as drawn in green. If the genus is two, this curve passes twice through one of the holes
  \end{itemize}

  To choose the filling collection if $\gamma_i$ or $\rho_i$ is a separating curve, arrange the holes of the surface in two columns, one on each side of the separating curve. The collection contains the following simple closed curves:
  
  \begin{itemize}
   \item The separating curve, drawn in thick black
   \item A curve around each hole and around the perimeter as drawn in blue and black
   \item A row of non-separating curves connecting the holes in each column and the perimeter, as drawn in grey and red
   \item A non-separating curve that intersects the separating curve in two points and passes through one hole on each side of the separating curve, as drawn in green.
  \end{itemize}

  By construction, condition (1) is satisfied. The collections chosen on a surface of even genus with respect to a non-separating curve and with respect to a curve that divides the surface exactly in half have the same full preimage in the two-fold cover described in Theorem \ref{classification} (and illustrated in Figure \ref{fillings}). Thus, condition (2) is satisfied. 
  
  By Sageev's construction, each filling collection of curves on a hyperbolic surface yields a CAT$(0)$ cube complex on which the surface group acts properly discontinuously and cocompactly by isometries. Since each curve is embedded and at most two distinct curves pairwise-intersect, the resulting cube complex is $2$-dimensional. Moreover, each resulting cube complex structure on the surfaces $S$ and $T$ is special, which can be seen as follows. The filling set of curves is in one-to-one correspondence with the set of hyperplanes of the resulting cube complex. The surfaces are orientable, so the hyperplanes are two-sided. Since the curves are embedded, the hyperplanes are embedded. Each filling set of  curves specified decomposes the surface into a cell complex of twelve polygons. A hyperplane osculates if and only if its corresponding curve lies along non-adjacent sides of one of the cells. This behavior does not occur in the cube complexes specified. Finally, two hyperplanes inter-osculate if and only if the two corresponding curves intersect and also lie along non-adjacent sides of one of the cells. As before, this behavior does not occur in the cube complexes specified. Thus, the resulting cube complex is special. 
  
  Take the barycentric subdivision of each cube complex constructed to obtain a finer two-dimensional non-positively curved special cube complex. Now, each of the amalgamating curves $\gamma_i$ and $\rho_i$ is a locally geodesic path of length eight in the $1$-skeleton of the cube complex. If the amalgamating curve is non-separating, there exists one vertex on this path that lies along the perimeter curve, and if the amalgamating curve is separating, there are two vertices on this path that lie along the perimeter curve. Identify these locally geodesic paths by a cubical isometry so that a vertex on the perimeter curve on $S$ is identified to a vertex on the perimeter curve on $T$.  By construction, Gromov's link condition holds after gluing, so the resulting complex is non-positively curved. 
  
  Examine the hyperplanes in the cube complex structure on $X_i$ obtained after gluing $S$ to $T$ to see that the complex is special. Restricted to each (orientable) surface $S$ or $T$, each hyperplane in $X_i$ lies parallel to one of the simple closed curves specified, so, the hyperplanes in the union are $2$-sided. Since the cube complex structures on $S$ and $T$ are special, to verify that the hyperplanes in the union do not self-intersect, osculate, or inter-osculate, it suffices to consider the hyperplanes that lie in both $S$ and $T$. If both amalgamating curves are non-separating or if both amalgamating curves are separating, then the number of hyperplanes restricted to each surface $S$ and $T$ does not decrease after gluing. Thus, in this case, the resulting cube complex is special. Otherwise, if a non-separating curve is glued to a separating curve, then on the surface glued along a separating curve, each of the two hyperplanes parallel to the perimeter curve on this surface is glued to two hyperplanes on the other surface. That is, on the surface glued along a non-separating curve, the hyperplanes parallel to the perimeter curve and parallel to a curve around one genus (as drawn in blue in Figure \ref{fillings}) become part of one hyperplane in the union $X_i$. So, the number of hyperplanes restricted to the surface glued along the non-separating curve decreases. Nonetheless, by construction, the resulting complex is special, proving claim (3). 
  
  Finally, by condition (2), the cube complex structure on $X_i$ and $X_j$ have the same full pre-image in the $2$-fold cover $Y$ for all $i, j$. Thus, the universal covers of the cube complexes are isomorphic. Therefore, each group in $H(\mathcal{G})$ acts properly discontinuously and cocompactly by isometries on the same CAT$(0)$ cube complex, with each quotient a $2$-dimensional special non-positively curved cube complex. 
  \end{proof}

  \begin{corollary}
   {\it If $G_1, G_2 \in \mathcal{C}_S$ and $G_1$ and $G_2$ are abstractly commensurable, then $G_1$ and $G_2$ act properly discontinuously and cocompactly by isometries on the same $2$-dimensional CAT$(0)$ cube complex with each quotient a non-positively curved special cube complex.   }
  \end{corollary}


 \bibliographystyle{alpha}
\bibliography{ACQI.bib}

\end{document}